\documentstyle[12pt]{article}
\newtheorem{theorem}{Theorem}  
\newtheorem{lemma}[theorem]{Lemma}
\newtheorem{proposition}[theorem]{Proposition}

\newtheorem{corollary}[theorem]{Corollary}
\newtheorem{remar}[theorem]{Remark}
\newenvironment{proof}{Proof:\ \ \ }{\QED}
\newenvironment{remark}{\begin{remar}\rm}{\end{remar}}

\newcommand{\QED}{{\unskip\nobreak\hfil\penalty50%
\hskip1em\hbox{}\nobreak\hfil $\Box$%
\parfillskip=0pt \finalhyphendemerits=0 \par\medskip\noindent}}
\newcommand{\bfind}[1]{\index{#1}{\bf #1}}

\newcommand{\n}{\par\noindent}

\newcommand{\sn}{\par\smallskip\noindent}
\newcommand{\mn}{\par\medskip\noindent}
\newcommand{\bn}{\par\bigskip\noindent}
\newcommand{\pars}{\par\smallskip}
\newcommand{\parm}{\par\medskip}
\newcommand{\parb}{\par\bigskip}

\newcommand{\co}{\mbox{\rm co}\,}
\newcommand{\ovl}[1]{\overline{#1}}
\newcommand{\ul}[1]{\underline{#1}}

\newcommand{\iT}{^{[i]}}

\newcommand{\Reg}{\mbox{\rm Reg}\,}
\newcommand{\subsetuneq}{\mathrel{\raisebox{.8ex}{\footnotesize%
$\displaystyle\mathop{\subset}_{\not=}$}}}
\newcommand{\toradresse}{\par\bigskip \small\rm
 Mathematical Sciences Group, 
 University of Saskatchewan, \par
 106 Wiggins Road, 
 Saskatoon, Saskatchewan, Canada S7N 5E6 \par
 email: fvk@math.usask.ca \ \ --- \ \ home page:
 http://math.usask.ca/$\,\tilde{ }\,$fvk/index.html}
\font\tenlv=msbm10 scaled 1200
\font\sevenlv=msbm7 scaled 1200
\font\fivelv=msbm5 scaled 1200
\def\lv #1{{\mathchoice{{\hbox{\tenlv #1}}}{{\hbox{\tenlv #1}}}
{{\hbox{\sevenlv #1}}}{{\hbox{\fivelv #1}}}}}

\newcommand{\N}{\lv N}

\newcommand{\R}{\lv R}
\newcommand{\Z}{\lv Z}
\newcommand{\F}{\lv F}

\begin{document}
\begin{center}
{\large\bf Maps on ultrametric spaces, Hensel's Lemma, and differential
equations over valued fields\footnote{I thank Lou van den
Dries for inviting me to Urbana and for making the manuscript [D]
available to me. This manuscript was the main inspiration for the
results of Section~\ref{sectde}. I also thank Florian Pop for the
key idea in the proof of Theorem~\ref{multhens}, and the referee for his
thorough reading, many useful suggestions, and the proof of
Lemma~\ref{D1=0}. This paper has undergone a major revision during my
stay at the Newton Institute at Cambridge; I gratefully
acknowledge its support.}}\\[.3cm]
{\it Franz-Viktor Kuhlmann}\\[.3cm]
15.~4.~2009
\end{center}
\begin{quote}
\small\rm
{\bf Abstract.}
We give a criterion for maps on ultrametric spaces to be
surjective and to preserve spherical completeness. We show how Hensel's
Lemma and the multi-dimensional Hensel's Lemma follow from our result.
We give an easy proof that the latter holds in every henselian field. We
also prove a basic infinite-dimensional Implicit Function Theorem.
Further, we apply the criterion to deduce various versions of Hensel's
Lemma for polynomials in several additive operators, and to give a
criterion for the existence of integration and solutions of certain
differential equations on spherically complete valued differential
fields, for both valued D-fields in the sense of Scanlon, and
differentially valued fields in the sense of Rosenlicht. We modify the
approach so that it also covers logarithmic-exponential power series
fields. Finally, we give a criterion for a sum of spherically complete
subgroups of a valued abelian group to be spherically complete. This in
turn can be used to determine elementary properties of power series
fields in positive characteristic.
\end{quote}

{\footnotesize\rm
\tableofcontents
}
%
%
\section{Introduction}
Hensel's Lemma (see Theorem~\ref{HL}) is an important tool in the theory
of valued fields. In recent years, it has witnessed several
generalizations. For example, such generalizations are important when
the valued fields are enriched by additional structure like derivations.
But attempts have also been made to formulate Hensel's Lemma in
situations with less structure. For instance, forgetting about
multiplication one may consider valued abelian groups or modules.
Another interesting case is that of a non-commutative multiplication.

In view of these developments, it is logical to ask for the underlying
principle that makes Hensel's Lemma work. This principle should be
formulated using as little algebraic structure as possible so that one
can derive new versions of Hensel's Lemma by adding whatever structure
one is interested in.

It has turned out that the structure suitable for such an
underlying principle is that of ultrametric spaces. In [P2], S.\
Prie{\ss}-Crampe proved an ultrametric Fixed Point Theorem. This theorem
works with contracting maps, and indeed the Newton algorithm used to
prove Hensel's Lemma for the field of $p$-adic numbers readily provides
such a map. But in other situations, contracting maps are not always
instantly available. For example, if one looks for zeros of polynomial
maps on a valued field, it can be more convenient to directly study
the ultrametric properties of these maps. The problem could then be
solved by showing surjectivity of such maps when restricted to suitable
subsets of the field. Our Ultrametric Main Theorem (Theorem~\ref{MT}) is
of this nature.

In the next section, we give a quick introduction to the facts about
ultrametric spaces that are necessary to understand the Ultrametric Main
Theorem. In Section~\ref{sectappl} we will then give a summary of the
various applications that are derived in this paper.

%
%
\subsection{The Ultrametric Main Theorem} \label{sectuv}
Let $(Y,u)$ be an ultrametric space. That is, $u$ is a map from $Y\times
Y$ onto a totally ordered set $\Gamma$ with last element $\infty$,
satisfying that for all $x,y,z\in Y$,\sn
(U1) \ $u(y,z)=\infty$ \ if and only if \ $y=z$,\n
(U2) \ $u(y,z)\geq\min\{u(y,x),u(x,z)\}$ \ \ \
(ultrametric triangle law),\n
(U3) \ $u(y,z)=u(z,y)$ \ \ \ (symmetry).\sn
It follows that\sn
$\bullet \ u(y,z)>\min\{u(y,x),u(x,z)\}\>\Rightarrow\>u(y,x)=u(x,z)$,\n
$\bullet \ u(y,x)\ne u(x,z)\>\Rightarrow\>u(y,z)=
\min\{u(y,x),u(x,z)\}$.\sn
We will use these properties freely.
We set $uY:=\{u(y,z)\mid y,z\in Y,y\ne z\}=\Gamma\setminus\{\infty\}$
and call it the \bfind{value set of} $(Y,u)$.
\pars
We recall some definitions. For $y\in Y$ and $\alpha\in uY\cup
\{\infty\}$, we define the \bfind{closed ball} around $y$ with radius
$\alpha$ as follows:
\[B_{\alpha}(y)\>:=\> \{z\in Y\mid u(y,z)\geq\alpha\}\;.\]
To facilitate notation, we will also use
\[B(x,y)\>:=\>B_{u(x,y)}(x)\;.\]
It follows from the ultrametric triangle law that $B_{u(x,y)}(x)=
B_{u(x,y)}(y)$ and that $B(x,y)$ is the smallest closed ball containing
$x$ and $y$. Similarly, it follows from the ultrametric triangle law that
\begin{equation}                            \label{umball}
B(x,y)\subseteq B(z,t) \ \ \mbox{ if and only if } \ \
x\in B(z,t) \mbox{ and } u(x,y)\geq u(z,t)\;.
\end{equation}
(Note: the bigger $u(x,y)$, the closer $x$ and $y$; this is compatible
with the Krull notation of valuations.)

A \bfind{ball} is the union of any non-empty collection of closed balls
which contain a common element. If $B_1$ and $B_2$ are balls with
non-empty intersection, then $B_1\subseteq B_2$ or $B_2\subseteq B_1\,$.

A set of balls in $(Y,u)$ is called a \bfind{nest of balls} if it is
totally ordered by inclusion; this is the case as
soon as every two balls in the set have a nonempty intersection. The
\bfind{intersection} of the nest is defined to be the intersection of
all of its balls. If it is non-empty, then it is again a ball.

The ultrametric space $(Y,u)$ is called \bfind{spherically complete} if
every nest of balls has a nonempty intersection. It is well known and
easy to prove that this holds if and only if every nest of closed balls
has a nonempty intersection. If $(Y,u)$ is spherically complete and $B$
is a ball in $Y$, then also $(B,u)$ is spherically complete.

\parm
Let $(Y,u)$ and $(Y',u')$ be non-empty ultrametric spaces and $f:\;Y
\rightarrow Y'$ a map. For $y\in Y$, we will write $fy$ instead of
$f(y)$. An element $z'\in Y'$ is called \bfind{attractor for $f$}
if for every $y\in Y$ such that $z'\ne fy$, there is an element $z\in Y$
which satisfies:
\sn
{\bf (AT1)} \ $u'(fz,z')>u'(fy,z')$,\n
{\bf (AT2)} \ $f(B(y,z))\subseteq B(fy,z')$.
\sn
Condition (AT1) says that the approximation $fy$ of $z'$ from within the
image of $f$ can be improved, and condition (AT2) says that this can be
done in a somewhat continuous way.

The following are our main theorems.

\begin{theorem}                             \label{MTattr}
Assume that $z'\in Y'$ is an attractor for $f:\;Y \rightarrow Y'$ and
that $(Y,u)$ is spherically complete. Then $z'\in f(Y)$.
\end{theorem}

The map $f$ will be called \bfind{immediate} if every $z'\in Y'$ is an
attractor for $f$.

\begin{theorem}                             \label{MT}
Assume that $f:\;Y \rightarrow Y'$ is immediate and that
$(Y,u)$ is spherically complete. Then $f$ is surjective and $(Y',u')$ is
spherically complete. Moreover, for every $y\in Y$ and every ball $B'$
in $Y'$ containing $fy$, there is a ball $B$ in $Y$ containing $y$
and such that $f(B)=B'$.
\end{theorem}
This theorem is a generalization of a result proved in [KU1] for
additive maps on spherically complete abelian groups (see
Section~\ref{sectvg} for the definition). Theorem~\ref{MT} also works in
the case where the map $f$ is not additive (or even when there is no
addition at all). It is related to ultrametric fixed point theorems as
proved in [P2], [PR1]. Compared to them, it has the advantage that it
can be applied to situations where a natural contracting map is not at
hand. There is also a variant of our ``Attractor Theorem''
(Theorem~\ref{MTattr}) which works for ultrametric spaces with partially
ordered value sets ([PR2]). For further information and applications of
ultrametric fixed point theorems, see also [SCH] and [PR3].

\pars
If $f$ is just the embedding of an ultrametric subspace $Y$ in an
ultrametric space $Y'$, then (AT2) will automatically hold. Hence, we
will say that $Y$ is an \bfind{immediate subspace of} $Y'$ if it is an
ultrametric subspace of $Y'$ and for all $z'\in Y'$ and $y\in Y$ there
is $z\in Y$ such that $u'(z,z')>u'(y,z')$. Now Theorem~\ref{MT} yields:

\begin{corollary}
Assume that $Y$ is an immediate ultrametric subspace of $Y'$. If
$(Y,u)$ is spherically complete, then $Y=Y'$.
\end{corollary}
\sn
It should be noted that an immediate subspace is not necessarily a dense
subspace.

\parm
A subspace $Y$ of $Y'$ is said to have the \bfind{optimal approximation
property} ({\bf in} $Y'$) if for every $z'\in Y'$ there is $z\in Y$ such
that $u'(z,z')= \max\{u'(y,z')\mid y\in Y\}$. The element $z$ need not
be uniquely determined. If the set $\{u'(y,z')\mid y\in Y\}$ has no
maximum, then $z'$ is an attractor for the embedding of $Y$ in $Y'$. On
the other hand, if $z'\in Y$, then the maximum is $u(z',z')=\infty$.
Thus, Theorem~\ref{MTattr} yields:
\begin{corollary}                           \label{scoap}
Assume that $Y$ is an ultrametric subspace of $Y'$. If $(Y,u)$ is
spherically complete, then it has the optimal approximation property.
\end{corollary}

%
%
\subsection{Applications}            \label{sectappl}
$\bullet$ \ {\bf The Additive Main Theorem}
\sn
In some applications, the map $f$ is a homomorphism of abelian groups
and the ultrametric $u$ is induced by a group (or field) valuation (see
Section~\ref{sectvg} for definitions). With the presence of addition,
balls can be shifted additively to balls that contain $0$. In this way,
the criteria for immediate maps become much easier to formulate and to
check (see Proposition~\ref{ATC}). In Section~\ref{sectav} we will prove
the additive version of our Ultrametric Main Theorem
(Theorem~\ref{MTadd}), which works for homomorphisms.

In Section~\ref{sectbc} we will introduce the notion of {\it
pseudo-companion} for arbitrary maps on valued abelian groups. One can
think of it as a linearization at a certain point ``up to terms of
higher order'', valuation theoretically speaking. This notion will then
play an essential role when we study polynomial maps.

\bn
$\bullet$ \ {\bf Hensel's Lemma revisited}
\sn
Let $(K,v)$ be a valued field with valuation ring ${\cal O}$ and
valuation ideal ${\cal M}$. Further, take a polynomial $f\in {\cal O}
[X]$ and $b\in {\cal O}$ such that $s:=f'(b)\ne 0$. In
Section~\ref{sectpm} we consider $f$ as a map on $K$ and prove that
{\it $f$ induces an immediate injective map from $b+s{\cal M}$ into
$f(b)+s^2{\cal M}$} (Proposition~\ref{findpl}). Here, the
pseudo-companion is simply multiplication by $s$. From Theorem~\ref{MT}
we obtain that {\it if $(K,v)$ is spherically complete (i.e., its
underlying ultrametric is spherically complete), then this map is onto}
(Theorem~\ref{MT2}).

This allows a new look at Hensel's Lemma: while it is always true for
$(K,v)$ spherically complete and $f'(b)\ne 0$ that the above map is
onto, the condition ``$vf(b)\geq 2vf'(b)$'' of Hensel's Lemma guarantees
that $0\in f(b)+s^2{\cal M}$ and consequently, there is $a\in K$ such
that $f(a)=0$ and $v(a-b)>vf'(b)$ (see Section~\ref{sectHLr}). We
generalize this result to systems of $n$ polynomials in $n$ variables
and use it to prove that the multidimensional Hensel's Lemma holds in
every spherically complete valued field (Theorem~\ref{sphcomHL}). By
an easy argument due to F.~Pop, we conclude that the multidimensional
Hensel's Lemma holds in every henselian field (see
Theorem~\ref{multhens}). Further, we prove results on the surjectivity
of functions defined by power series in spherically complete valued
fields (see Section~\ref{sectpsf}).

Our above approach to Hensel's Lemma has also been used in a
non-commutative setting. In [VC] it is applied to skew power series
fields over skew fields.

\bn
$\bullet$ \ {\bf Infinite-dimensional Implicit Function Theorems}
\sn
The $n$-fold product of a spherically complete ultrametric space is
again spherically complete (see Section~\ref{sectpd}). We use this fact
for the proof of the multi-dimensional Hensel's Lemma. If one thinks of
generalizing this to an infinite-dimensional version, one runs into
problems when trying to define a suitable product. But if one restricts
the scope to valued rings with well ordered value sets, then this is
possible. Using the above mentioned notion of {\it pseudo-companion}, we
formulate in Sections~\ref{sectidIFT} and~\ref{sectpsiift} several
infinite-dimensional Implicit Function Theorems, for polynomial and
power series maps. Such theorems are of interest for B.\ Teissier's
approach to local uniformization in arbitrary characteristic (cf.\ [T],
Theorem~5.56).

\bn
$\bullet$ \ {\bf VD-fields}
\sn
A VD-field is a valued field $(K,v)$ with an additive map
$D: K \rightarrow K$ satisfying conditions that are a relaxation of
T.~Scanlon's axioms for valued D-fields (cf.~[S1,2]). Scanlon's notion
comprises both differential and
difference fields. Essential features of VD-fields are that the
value $vDa$ depends on the value $va$ in a sufficiently simple way and
that $D$ induces an additive map on the residue field of $K$ (again
denoted by $D$). The following result, proved in Section~\ref{sectdf},
shows that in this setting, the notion of {\it immediate map}
appears in a very natural way: {\it If $(K,D,v)$ is a VD-field,
then $D$ is immediate if and only if $D$ is surjective on $Kv$}
(Theorem~\ref{Dimmsurj}). Hence we obtain from Theorem~\ref{MT} that
{\it if $(K,D,v)$ is a spherically complete VD-field such that
$D$ is surjective on $Kv$, then $D$ is surjective on $K$} (see
Theorem~\ref{DF1}).

In Section~\ref{sectdf} we will also prove the following version of
Scanlon's D-Hensel's Lemma (cf.~[S1,2]). By $D^i$ we denote the $i$-th
iterate of $D$. The residue field $Kv$ is said to be \bfind{linearly
$D$-closed} if each operator $\,\sum_{i=0}^{n}c_iD^i$ with $c_i\in Kv$
is surjective on $Kv$.
\begin{theorem}                             \label{DHL}
Let $(K,D,v)$ be a spherically complete VD-field whose residue
field is linearly $D$-closed. Take a polynomial $f\in {\cal O}
[X_0,X_1,\ldots,X_n]$ and assume that there is some $b\in {\cal O}$
such that
\[\gamma\>:=\>\min_{0\leq i\leq n} v\frac{\partial f}{\partial X_i}
(b, Db,\ldots, D^nb) \><\> \infty \mbox{ \ \ \ and \ \ \ }
vf(b,Db\ldots,D^nb)>2\gamma\;.\]
Then there is an element $a\in K$ such that $f(a,Da,\ldots,D^na)
=0$ and $v(a-b)>\gamma$.
\end{theorem}

In fact, we will deduce this theorem from a much more general
Hensel's Lemma for polynomials in several additive operators
(Theorem~\ref{genHLsao} in Section~\ref{sectwcm}).

\bn
$\bullet$ \ {\bf Rosenlicht valued differential fields}
\sn
A valuation $v$ on a differential field $(K,D)$ is a {\it differential
valuation} in the sense of M.~Rosenlicht (cf.\ [R1]) if it satisfies an
axiom that is derived from de l'H\^opital's Rule. In this case, there is
in general no simple correspondence between the values $vDa$ and $va$,
and there is also no suitable map induced on the residue field. Yet
again, immediate maps appear naturally. We say that $(K,D)$
\bfind{admits integration} if $D$ is surjective, and that $(K,D,v)$
\bfind{admits asymptotic integration} (cf.\ [R2]) if for every $a'\in
K\setminus \{0\}$, there is some $a\in K$ such that
\[v(a'-Da)\> >\> va'\;.\]
In Section~\ref{sectrvdf}, we will give the (easy) proof of the
following fact: {\it If $v$ is a differential valuation on $(K,D)$, then
$D$ is immediate if and only if $(K,D,v)$ admits asymptotic integration}
(see Proposition~\ref{dv=tc}). Hence we obtain from Theorem~\ref{MT}:
{\it Let $(K,D)$ be a differential field, endowed with a
spherically complete differential valuation $v$. If
$(K,D,v)$ admits asymptotic integration, then $(K,D)$ admits
integration} (Theorem~\ref{D1}).

\pars
In Section~\ref{sectrvdf} we will also prove a theorem about integration
on the union of an increasing chain of spherically complete Rosenlicht
valued differential fields (Theorem~\ref{LE}). It can be used to
show that the derivation on the logarithmic-exponential power series
field $\R((t))^{LE}$ (cf.\ [DMM3]) is surjective.

\pars
When we try to prove a ``differential Hensel's Lemma'' for Rosenlicht's
differential valuations, we experience technical problems because of the
weak correspondence between the values $vDa$ and $va$. In this case, the
results are not as nice and simple as in the case of VD-fields.
The main results are Theorem~\ref{DVHL}, obtained from the more general
Theorem~\ref{genHLdom} proved in Section~\ref{sectdom}, and
Theorem~\ref{DVHLR}, obtained from the more general Theorem~\ref{genHLRos}
proved in Section~\ref{sectRso}. As a simple application we obtain a
result which was proved by Lou van den Dries in~[D] (see
Corollary~\ref{VDD}).

\bn
$\bullet$ \ {\bf Sums of spherically complete valued abelian groups}
\sn
So far, we have been interested in the surjectivity of maps. Here is an
application where we use that the image of the map inherits spherical
completeness. It is used in [KU2] to determine elementary properties of
the power series field $\F_p((t))$ in connection with \bfind{additive
polynomials}. A polynomial $f$ is called additive on an infinite field
$K$ if $f(a+b)=f(a)+f(b)$ for all $a,b\in K\>$ (cf.\ [L], VIII, \S 11).
For example, the polynomials $X^p$ and $X^p-X$ are additive on
$\F_p((t))$ and every other field of characteristic $p$. For every
additive polynomial $f$ on a field $K$, the image $f(K)$ is a subgroup
of the additive group of $K$. If $f_1,\ldots,f_n$ are additive
polynomials with coefficients in $K$, then the sum
$f_1(K)+\ldots+f_n(K)$ is again a subgroup of the additive group of $K$.

If $K$ is a maximally valued field (like $K=\F_p((t))\,$; cf.\
Section~\ref{sectvf}), then the image
$f(K)$ of every polynomial is spherically complete. Hence the question
arises whether the subgroup $f_1(K)+\ldots+f_n(K)$ is again spherically
complete. In Section~\ref{sectgp} we will show that the sum of
spherically complete subgroups of a valued abelian group is spherically
complete (and hence has the optimal approximation property) if the sum
is {\it pseudo-direct} (cf.\ Theorem~\ref{addgr}). The optimal
approximation property of a definable subgroup in a valued abelian group
is an elementary property in the language of groups with a predicate for
the valuation. If the subgroups are definable, then also the assertion
that their sum is pseudo-direct is elementary. Hence, given additive
polynomials $f_1,\ldots,f_n$ with coefficients in $K=\F_p((t))$, the
assertion
\begin{center}
\it if $f_1(K)+\ldots+f_n(K)$ is pseudo-direct, then it has the
optimal approximation property
\end{center}
is elementary in the language of valued fields (enriched by names for
the coefficients of the polynomials $f_i$). By Theorem~\ref{addgr}, it
holds for $K=\F_p((t))$, and for every other spherically complete
valued field $(K,v)$. See [KU2] and [KU3] for further details.

%
%
\section{Ultrametric Spaces}
%
%
%
\subsection{Proof of the Ultrametric Main Theorem}  \label{sectpuv}
For the proof of Theorem~\ref{MTattr}, we show the following more
precise statement:
\begin{lemma}
Assume that $z'\in Y'$ is an attractor for $f:\;Y \rightarrow Y'$ and
that $(Y,u)$ is spherically complete. Then for every $y\in Y$ there is
$z_0\in Y$ such that $fz_0=z'$ and $f(B(y,z_0))\subseteq B(fy,z')$.
\end{lemma}
\begin{proof}
If $z'=fy$ then we set $z_0=y$ and there is nothing to show. So assume
that $z'\ne fy$. Then by assumption on $z'$ there is $z\in Y$ such that
(AT1) and (AT2) hold. Take elements $y_i,z_i\in B(y,z)$, $i\in I$, such
that the balls $B(y_i,z_i)$ form a nest inside of $B(y,z)$, maximal with
the following properties, for all~$i$:
\sn
i) \ $z'=fy_i=fz_i\>$ or $\>u'(z',fz_i)>u'(z',fy_i)$,\n
ii) \ $f(B(y_i,z_i))\subseteq B(fy_i,z')$,\n
iii) \ for all $j\in I$, $u(y_i,z_i)<u(y_j,z_j)$ implies that
$u'(fy_i,z')<u'(fy_j,z')$.
\sn
Non-empty nests with these properties exist. Indeed, the singleton
$\{B(y,z)\}$ is such a nest. Maximal nests with these properties
exist by Zorn's Lemma. Take one such maximal nest. As soon as we find
$z_0\in B(y,z)$ such that $z'=fz_0$ we are done because
$f(B(y,z_0))\subseteq f(B(y,z)) \subseteq B(fy,z')$.

Assume first that this nest has a minimal ball, say, $B(y_0,z_0)$. If
$z'=fz_0$ then we are done. So assume that $z'\ne fz_0$, and set
$\tilde{y}:= z_0\,$. Then by assumption on $z'$, we can find
$\tilde{z}\in Y$ such that
\[u'(f\tilde{z},z')>u'(f\tilde{y},z') \mbox{ \ \ and \ \ }
f(B(\tilde{y},\tilde{z}))\subseteq B(f\tilde{y},z')\;.\]
We have that
\begin{equation}                            \label{uztz0}
u'(f\tilde{y},z')=u'(fz_0,z')>u'(fy_0,z')=u'(f\tilde{y},fy_0)\;,
\end{equation}
where the last equality follows from the ultrametric triangle law.
So we know that $fy_0\notin B(f\tilde{y},z')$ and thus, $y_0\notin
B(\tilde{y},\tilde{z})$. This shows that $u(\tilde{y},\tilde{z})>
u(\tilde{y},y_0)=u(z_0,y_0)$, and since
$\tilde{y}= z_0\in B(z_0,y_0)$, it follows that
$B(\tilde{z},\tilde{y}) \subsetuneq B(z_0,y_0)$. So we
can enlarge our nest of balls by adding $B(\tilde{z},\tilde{y})$, and
conditions i) and ii) hold for the new nest. From iii) we see that
$u'(fy_0,z')$ is maximal among the $u'(fy_i,z')$, $i\in I$; so
(\ref{uztz0}) shows that also iii) holds for the new nest. But this
contradicts the maximality of the chosen nest.

Now assume that the nest contains no smallest ball. Since $(Y,u)$ is
spherically complete by assumption, there is some $z_0\in \bigcap_{i\in
I} B(y_i,z_i)$. Suppose that $fz_0\ne z'$. Then we set $\tilde{y}:=
z_0\,$. For all $i$, we have $\tilde{y}\in B(y_i,z_i)$ and
$f\tilde{y}\in f(B(y_i,z_i)) \subseteq B(fy_i,z')$, showing that
$u'(f\tilde{y},z')\geq u'(fy_i,z')$. We choose $\tilde{z}$ as before. We
have $f(B(\tilde{y},\tilde{z})) \subseteq B(f\tilde{y},z') \subseteq
B(fy_i,z')$ for all $i$. On the other hand, since the nest contains no
smallest ball, the set $\{u(y_i,z_i)\mid i\in I\}$ has no maximal
element. So iii) implies that also the set $\{u'(fy_i,z')\mid i\in I\}$
has no maximal element. Consequently, for all $i\in I$ there is $j\in I$
such that $u'(f\tilde{y},z')\geq u'(fy_j,z')>u'(fy_i,z')\,$.
Consequently, $fy_i\notin B(f\tilde{y},z')$, which yields that
$y_i\notin B(\tilde{y},\tilde{z})$. Therefore, $B(\tilde{y},\tilde{z})
\subsetuneq B(y_i,z_i)$ and $u(\tilde{y},\tilde{z})>u(y_i,z_i)$ for all
$i$. So we can enlarge our nest of balls by adding
$B(\tilde{y},\tilde{z})$, and conditions i), ii) and iii) hold for the
new nest. This again contradicts the maximality of the chosen nest.
Hence, $fz_0=z'$ and we are done.
\end{proof}

\begin{corollary}                           \label{BB}
Assume that $f:\;Y \rightarrow Y'$ is immediate and that
$(Y,u)$ is spherically complete. Then the following holds:
\sn
{\bf (BB)} \
for every $y\in Y$ and every ball $B'$ in $Y'$ around $fy$, there is
a ball $B$ in $Y$ around $y$ such that $f(B)=B'$.
\end{corollary}
\begin{proof}
Assume that $y\in Y$ and that $B'$ is any ball in $Y'$ which
contains $fy$. Then we can write
\[B'=\bigcup_{z'\in B'} B(z',fy)\;.\]
According to the foregoing lemma, for every $z'$ there is
$z_0\in Y$ such that $z'\in f(B(y,z_0))\subseteq
B(fy,z')\subseteq B'$. Take $B$ to be the union over all such
balls $B(y,z_0)$ when $z'$ runs through all elements of
$B'$. Then $B$ is a ball around $y$ satisfying $f(B)=B'$.
\end{proof}

The next lemma proves Theorem~\ref{MT}:

\begin{lemma}
Assume that $f:\;Y \rightarrow Y'$ is a map which satisfies (BB), and
that $(Y,u)$ is spherically complete. Then $f$ is surjective, and
$(Y',u')$ is spherically complete.
\end{lemma}
\begin{proof}
Taking $B'=Y'$, we obtain the surjectivity of $f$.

Now we take any nest of balls $\{B'_j\mid j\in J\}$ in $Y'$. We have to
show that this nest has a nonempty intersection. We claim that in $Y$
there exists a nest of balls $B_i$, $i\in I$, maximal with the property
that
\begin{equation}                            \label{IJ}
\mbox{$I\subseteq J$, and for all $i\in I$, $f(B_i)=B'_i\>$.}
\end{equation}
To show this, we first take any $j\in J$ and choose some $y_j\in Y$
such that $fy_j\in B'_j\,$, making use of the surjectivity of $f$. As
$f$ satisfies (BB), we can choose a ball $B_j$ in $Y$ around $y_j$ and
such that $f(B_j)=B'_j\,$. So the nest $\{B_j\}$ has property
(\ref{IJ}). Hence, a maximal nest $\{B_i\mid i\in I\}$ with property
(\ref{IJ}) exists by Zorn's Lemma.

We wish to show that the balls $B'_i\,$, $i\in I$, are coinitial in the
nest $B'_j\,$, $j\in J$, that is, for every ball $B'_j$ there is some
$i\in I$ such that $B'_i\subseteq B'_j\,$. Once we have shown this we
are done: as $Y$ is spherically complete, there is some $y\in
\bigcap_{i\in I} B_i$, and
\[fy\,\in\, \bigcap_{i\in I} f(B_i) \>=\> \bigcap_{i\in I} B'_i
\>=\> \bigcap_{j\in J} B'_j\]
shows that $\bigcap_{j\in J} B'_j$ is non-empty.

Suppose the balls $B'_i\,$, $i\in I$, are not coinitial in the nest
$B'_j\,$, $j\in J$. Then there is some $j\in J$ such that $B'_j
\subsetuneq B'_i$ for all $i\in I$. Since $Y$ is spherically complete,
there is some $y\in \bigcap_{i\in I} B_i\,$. We have that $fy\in
\bigcap_{i\in I} B'_i=:B'$, and also that $B'_j \subseteq B'$. By
assumption, there is a ball $B$ around $y$ such that
$f(B)=B'$. If $B'$ happens to be the smallest ball among the $B'_i\,$,
say, $B'=B'_{i_0}$ with $i_0\in I$, then we just take $B=B_{i_0}\,$. If
$B'\subsetuneq B'_i\,$, then it follows that $B\subsetuneq B_i\,$. Hence
in all cases, $B\subseteq B_i$ for all $i$. Since $B'_j \subseteq B'$,
we can choose $\tilde{y}\in B$ such that $f\tilde{y}\in B'_j\,$. By
assumption, there is a ball $B_j$ around $\tilde{y}$ such that
$f(B_j)=B'_j\,$. Since $\tilde{y}\in B_i$ for all $i\in I$, we know that
$B_i\,$, $i\in I\cup\{j\}$ is a nest of balls. By construction, it has
property (\ref{IJ}). Since $j\notin I$, this contradicts our maximality
assumption on $I$. This proves that the balls $B'_i\,$, $i\in I$, must
be coinitial in the nest $B'_j\,$, $j\in J$.
\end{proof}

%
%
\subsection{Products}                       \label{sectpd}
Let $(Y_i,u_i)$, $i\in I$, be ultrametric spaces whose value sets
$u_i Y_i$ are all contained in a common ordered set, and assume that
$I$ is finite or that $\bigcup_{i\in I} u_i Y_i$ is well ordered. Then
their {\bf direct product}\index{product of ultrametric spaces}
will be the cartesian product $\prod_{i\in I} Y_i$ equipped with the
ultrametric
\[u:\>\prod_{i\in I} Y_i\times\prod_{i\in I}
Y_i \,\rightarrow\, \bigcup_{i\in I} u_i Y_i\cup\{\infty\}\]
defined by
\[u\,((y_i)_{i\in I}\,,\,(z_i)_{i\in I}):=\min_{i\in I}
u_i(y_i,z_i)\;.\]
We leave it to the reader to verify that this map satisfies (U1), (U2)
and (U3). Note that indeed every element of $\bigcup_{i\in I} u_i Y_i$
appears as the distance of two suitably chosen elements of
$\prod_{i\in I} Y_i\,$.

\begin{lemma}
Take $k\in I$ and let $\pi_k:\prod_{i\in I} Y_i\rightarrow Y_k$ denote
the
projection onto the $k$-th component. If $B$ is a ball in $(\prod_{i\in
I} Y_i,u)$, then for every $k\in I$, $\pi_k B$ is a ball in $(Y_i,u_i)$,
and
\begin{equation}                            \label{B=p}
B\>=\>\prod_{i\in I}\pi_i B\;.
\end{equation}
\end{lemma}
\begin{proof}
Since $B\ne\emptyset$, we have that $\pi_k B\ne\emptyset$ and we can
pick an element $y_k\in \pi_k B$ which is the projection of some
$y=(y_i)_{i\in I}\in B$. We claim that
\begin{equation}                            \label{B=u}
\pi_k B\>=\>\bigcup_{z\,\in\, B} B(y_k,\pi_k z)\;,
\end{equation}
where $B(y_k,\pi_k z)$ is understood to designate a ball in $(Y_k,u_k)$.
Since $\pi_k z\in B(y_k,\pi_k z)$, the inclusion ``$\subseteq$'' is
trivial. Now take $z=(z_i)_{i\in I}\in B$ and some $x_k\in B(y_k,\pi_k
z)$. Set $x=(x_i)_{i\in I}$ with $x_i:=y_i$ for $k\ne i\in I$. Then
$u(y,x)= u_k(y_k,x_k)\geq u_k(y_k,\pi_k z)\geq u(y,z)$ and therefore,
$x\in B$ and $x_k\in \pi_k B$. This proves that ``$\supseteq$'',
and hence equality holds in (\ref{B=u}). As a union of balls with common
element $y_k$, $\pi_k B$ is itself a ball.

The inclusion ``$\subseteq$'' in (\ref{B=p}) is trivial. For the
converse, pick an element $x=(x_i)_{i\in I}\in \prod_{i\in I}\pi_i B$.
Then there are elements $z^i\in B$ such that $x_i=\pi_i z^i$ for all
$i\in I$. Pick an arbitrary element $y\in B$. Then for some $j\in I$,
$u(y,x)=\min u_i(y_i,x_i)=\min u_i(y_i,\pi_i z^i)=u_j(y_j,\pi_j z^j)\geq
u(y,z^j)$. Since $y,z^j\in B$, it follows that $x\in B$. This proves
the inclusion ``$\supseteq$'' and hence equality in (\ref{B=p}).
\end{proof}

\begin{proposition}                               \label{prodsphc}
If the ultrametric spaces $(Y_i,u_i)$, $i\in I$, are
spherically complete, then the same holds for their
direct product $(\prod_{i\in I} Y_i\,,\,u)$.
\end{proposition}
\begin{proof}
Let ${\bf B} = \{B_j\mid j\in J\}$ be a nest of balls in the direct
product. We have to show that the intersection of {\bf B} is nonempty.
For every $i\in I$ we consider the projections $\pi_i B_j$ which by the
foregoing lemma are balls in $(Y_i,u_i)$. Since {\bf B} is a nest, all
intersections $B_j\cap B_k$ are non-empty and therefore, all
intersections $\pi_i B_j \cap \pi_i B_k$ are non-empty. This proves that
for each $i\in I$, $\{\pi_i B_j\mid j\in J\}$ is a nest of balls in
$(Y_i,u_i)$. By our assumption that the ultrametric spaces $(Y_i,u_i)$
are spherically complete, there exist elements $x_i\in \bigcap_{j\in J}
\pi_i B_j$ for each $i$. By equation (\ref{B=p}) of the foregoing lemma,
$(x_i)_{i\in I}\in B_j$ for every $j\in J$, hence $(x_i)_{i\in I}\in
\bigcap_{j\in J} B_j\,$.
\end{proof}

%
%
\subsection{Embeddings and isomorphisms}
Take ultrametric spaces $(Y,u)$ and $(Y',u')$ and a map
$f:Y\rightarrow Y'$. A map $\varphi: uY\rightarrow u'Y'$
will be called a \bfind{value map for $f$} if it preserves $\leq$ and
satisfies $u'(fy,fz)=\varphi u(y,z)$ for all $y,z\in Y$, $y\ne z$.
From the latter it follows that $f$ is injective since $u'(fy,fz)=
\varphi u(y,z)\in u'Y'$ means that $u'(fy,fz)\ne \infty$, i.e.,
$fy\ne fz$. We call $f$ an \bfind{embedding of ultrametric spaces
(with value map $\varphi$)} if in addition, $\varphi$ preserves $<$
and hence is itself injective. An embedding $f$ is called an
\bfind{isomorphism of ultrametric spaces} if it is onto.
In this case, also $\varphi$ is onto. We set $\varphi\infty=\infty$.

%
%
\section{Immediate maps on valued abelian groups}  \label{sectvg}
A \bfind{valued abelian group} $(G,v)$ is an abelian group $G$ endowed
with a \bfind{valuation} $v$. That is, $a\mapsto va$ is a map from $G$
onto $vG\cup\{\infty\}$, where $vG$ is a totally ordered set and
$\infty$ is an element bigger than all elements of $vG$, and the
following laws hold:\sn
{\bf (V1)} \ $va=\infty\Leftrightarrow a=0\,$,\n
{\bf (V2)} \ $v(a-b)\geq\min\{va,vb\}$ \ \ \ (ultrametric triangle
law).\sn
The \bfind{value set} of $(G,v)$ is $vG$.
For every valued abelian group $(G,v)$, the set $G$ endowed with the map
\[u:G\times G\rightarrow vG\cup\{\infty\}\>,\;\;\; u(a,b):=v(a-b)\]
is an ultrametric space. We note the following translations of
properties of the ultrametric:\sn
$\bullet \ v(a-b)>\min\{va,vb\}\>\Rightarrow\>va=vb$,\n
$\bullet \ va\ne vb\>\Rightarrow\>v(a-b)=\min\{va,vb\}$,\n
$\bullet \ va=v(-a)$.

\pars
A valued abelian group $(G,v)$ is called \bfind{spherically complete}
if the underlying ultrametric space $(G,u)$ is spherically complete.
Standard examples for spherically complete abelian groups are the
Hahn products (see, e.g., [KU4]).

\pars
Observe that in a valued abelian group, any ball around $0$ is a
subgroup. Since balls are unions of closed balls, this has only to be
proved for closed balls. Note that
\[B_\alpha(0)\>=\>\{z\in G\mid u(0,z)\geq\alpha\}\>=\>\{z\in G\mid
vz\geq\alpha\}\]
since $u(0,z)=v(0-z)=v(-z)=vz$. Take $a,b\in B_\alpha(0)$. Then
$va\geq\alpha$ and $vb\geq\alpha$, whence $v(a-b)\geq\alpha$ by (V2),
that is, $a-b\in B_\alpha(0)$. This proves that every $B_\alpha(0)$ and
every other ball $B$ containing $0$ is a subgroup of $G$. Let us note
that since every ball $B$ containing $0$ is a union of closed balls
$B_\alpha(0)$, it follows that
\[y\in B\mbox{ and } vz\geq vy\;\Rightarrow\; z\in B\;.\]

Every ball $\tilde{B}$ in $(G,v)$ can be written in the form $b+B$ where
$b\in \tilde{B}$ and $B=\{a-b\mid a\in\tilde{B}\}$ is a ball around $0$.
Hence the balls in $(G,v)$ are precisely the cosets with respect to the
subgroups that are balls.

%
%
\subsection{Immediate homomorphisms}            \label{sectav}
In this section we will give a handy criterion for group homomorphisms
to be immediate. Throughout, let $(G,v)$ and $(G',v')$ be valued abelian
groups.

\begin{proposition}                         \label{ATC}
Let $f:G\rightarrow G'$ be a map such that $f0=0$. If $f$ is immediate,
then for every $a'\in G'\setminus\{0\}$ there is some $a\in G$ such that
\sn
{\bf (IH1)} \ $v'(a'-fa)>v'a'$,\n
{\bf (IH2)} \ for all $b\in G$, $\;va\leq vb\>$ implies $\>v'fa\leq
v'fb\,$.
\sn
The converse is true if $f$ is a group homomorphism.
\end{proposition}
\begin{proof}
Suppose first that $f$ is immediate, and take any $a'\in G'$, $a'\ne 0$.
Set $z':=a'$ and $y:=0$. Take $z\in G$ such that conditions (AT1) and
(AT2) hold, and set $a:=z$. Then $v'(a'-fa)=u'(z',fz)>u'(z',fy)=v'(a'
-f0)= v'a'$. Hence, (IH1) holds. Also, we obtain from the ultrametric
triangle law that $v'a'=v'fa$. Further, condition (AT2) shows that
\begin{eqnarray*}
f(\{b\mid vb\geq va\}) & = & f(B(0,a))\>=\>f(B(y,z))\\
& \subseteq & B(fy,z')\>=\>B(0,a')\>=\>\{b'\mid v'b'\geq v'a'=v'fa\}\;.
\end{eqnarray*}
That is, $va\leq vb\Rightarrow v'fa\leq v'fb$, i.e., (IH2) holds.

\pars
For the converse, take any $y\in G$ and $z'\in G'\setminus\{fy\}$. Set
$a':=z'-fy\ne 0$. Choose $a\in G$ such that conditions (IH1) and (IH2)
hold, and set $z:=y+a$. Then $u'(z',fz)=v'(z'-fz)=v'(z'-fy-fa)=v'(a'-fa)
>v'a'=v'(z'-fy)=u'(z',fy)$. So (AT1) holds. Also, we obtain from the
ultrametric triangle law that $v'fa=v'(z'-fy)$. To show that (AT2)
holds, take any $x\in B(y,z)$. Then $v(x-y)\geq v(z-y)=va$. Hence by
(IH2), $v'(fx-fy)=v'f(x-y)\geq v'fa=v'(z'-fy)$, so $fx\in B(fy,z')$.
\end{proof}

By Theorem~\ref{MT}, we obtain:

\begin{theorem}                             \label{MTadd}
Let $f:G\rightarrow G'$ a group homomorphism which satisfies (IH1) and
(IH2). Assume further that $(G,v)$ is spherically complete. Then $f$ is
surjective and $(G',v')$ is spherically complete.
\end{theorem}

\begin{lemma}                               \label{fftilde}
Let $f,\tilde{f}:G\rightarrow G'$ be group homomorphisms. Suppose that
$f$ is immediate and for all $a\in G$,
\begin{equation}                            \label{vfftilde}
v'(\tilde{f}a-fa)\;>\;v'fa \mbox{ \ \ or \ \ } \tilde{f}a=fa=0\;.
\end{equation}
Then also $\tilde{f}$ is immediate.
\end{lemma}
\begin{proof}
If $f$ satisfies (IH1) of Proposition~\ref{ATC}, then $v'(a'-\tilde{f}a)
\geq \min \{v'(a'-fa),v'(\tilde{f}a-fa)\}>v'\tilde{f}a=v'a'$, showing
that also $\tilde{f}$ satisfies (IH1). Since (\ref{vfftilde}) implies
that $v'\tilde{f}a=v'fa$, $\tilde{f}$ will satisfy (IH2) whenever $f$
does. Hence by Proposition~\ref{ATC}, $\tilde{f}$ is immediate
whenever $f$ is.
\end{proof}


For an arbitrary map $f:G\rightarrow G'$ we will say that $a\in G$ is
\bfind{$f$-regular} if it is non-zero and satisfies condition (IH2). We
will denote the set of all $f$-regular elements by $\Reg(f)$. Then the
following holds:

\begin{proposition}                         \label{indreg}
If $f:G\rightarrow G'$ is an immediate group homomorphism, then
\[va\>\mapsto\>v'fa\]
for $a\in \Reg(f)$ induces a well defined and $\leq$-preserving map
from $\{va\mid a\in\Reg(f)\}$ onto $v'G'$.
\end{proposition}
\begin{proof}
If $a,b\in \Reg(f)$ such that $va=vb$, then by (IH2), $v'fa\leq v'fb$
and $v'fa\geq v'fb$, whence $v'fa=v'fb$. This shows that the map is
well defined. Again because of (IH2), it preserves $\leq$. Now take any
$a'\in v'G'$, $a'\ne 0$. Then by (IH1), there is $a\in G$ such that
$v'(a'-fa)>v'a'$, whence $v'a'=v'fa$ by the ultrametric triangle law.
This proves that the map is onto.
\end{proof}

%
%
\subsection{Basic criteria}            \label{sectbc}
Even if the map $f$ that we consider on a valued abelian group is not a
homomorphism, the presence of addition helps us to give handy and
natural criteria for the map to be immediate. We just have to work a
little harder. In this section, we present basic criteria that will
cover all our applications in the non-additive case.

\begin{proposition}                         \label{BCb}
Take valued abelian groups $(G,v)$ and $(G',v')$, an element $b\in G$, a
ball $B$ around $0$ in $G$, a ball $B'$ around $0$ in $G'$, and a map
$f:b+B\rightarrow fb+B'$. Assume that $\phi:B\rightarrow B'$ is a map
such that for all $a'\in B'\setminus\{0\}$ there is $a\in\Reg(\phi)$
with the following properties:
\begin{equation}                            \label{(BC1)}
v'(a'-\phi a)\>>v'a'\>=\>v'\phi a\;,
\end{equation}
and
\begin{equation}                            \label{(BC2)}
v'(fy-fz\,-\,\phi(y-z))\>>\>v'\phi a
\mbox{ \ for all $\>y,z\in b+B$ such that $\>v(y-z)\geq va\,$.}
\end{equation}
\mn
Then $f$ is immediate.
\pars
If $\phi 0=0$ then (\ref{(BC2)}) needs to be checked only for $y\ne z$.
\end{proposition}
\begin{proof}
Take $z'\in fb+B'$ and $y\in b+B$ such that $z'\ne fy$. Applying
our assumption to $a':=fy-z'$ we find that there is some
$a\in\Reg(\phi)$ such that by (\ref{(BC1)}),
\begin{equation}                            \label{v'-phib}
v'(fy-z'\,-\,\phi a)\>>\>v'(fy-z')\>=\>v'\phi a\;,
\end{equation}
and such that (\ref{(BC2)}) holds. Set $z:=y-a\in y-B=y+B=b+B$. Then
$y-z=a$ and hence by (\ref{(BC2)}) and (\ref{v'-phib}),
\[v'(fy-fz-\phi (y-z))\>>\>v'\phi a\>=\>v'(fy-z')\;.\]
Consequently,
\begin{eqnarray*}
v'(z'-fz) & \geq &
\min\{v'(z'-fy\,+\,\phi a)\,,\,v'(fy-fz-\phi a)\} \\
& = & \min\{v'(fy-z'\,-\,\phi a)\,,\,v'(fy-fz-\phi (y-z))\}  \\
 & > & v'(fy-z')\>=\>v'(z'-fy)\;.
\end{eqnarray*}
Hence (AT1) holds. Now take $x\in B(y,z)\subseteq b+B$, i.e., $v(y-x)
\geq v(y-z)=va$. Then $v'\phi(y-x) \geq v'\phi a$ because $a\in
\Reg(\phi)$, and $v'(fy-fx\,-\,\phi (y-x))>v'\phi a$ by (\ref{(BC2)}).
Therefore,
\[v'(fy-fx)\>\geq\>\max\{v'(fy-fx\,-\,\phi (y-x))\,,\,v'\phi (y-x)\}
\>\geq\>v'\phi a\>=\>v'(fy-z')\>,\]
whence $fx\in B(fy,z')$. Hence (AT2) holds.

Assume that $\phi 0=0$. Observe that $\phi a\ne 0$ since $a'\ne 0$ and
$v'a'=v'\phi a$. Hence if $y=z$ then $v'(fy-fz\,-\,\phi(y-z))=v'0=
\infty>v'\phi a$, which shows that (\ref{(BC2)}) need only be checked
for $y\ne z$.
\end{proof}
\n
Note that by the ultrametric triangle law, the equality in (\ref{(BC1)})
is a consequence of the inequality. Further, observe that this
proposition proves the direction ``$\Leftarrow$'' of
Proposition~\ref{ATC}: if we take $B=G$, $B'=G'$ and $\phi=f$, then
(IH1) implies (\ref{(BC1)}) and (IH2) implies that $a\in\Reg(\phi)$,
while (\ref{(BC2)}) is trivially satisfied. Hence if for every $a'\in
G'\setminus\{0\}$ there is $a\in G$ such that (IH1) and (IH2) hold, then
the above proposition shows that $f$ is immediate.

\parm
The following is a special case of the above criterion,
with nicer properties.

\begin{proposition}                               \label{rangepsd}
Take valued abelian groups $(G,v)$ and $(G',v')$,
an element $b\in G$, a ball $B$ in $G$ around $0$, a ball $B'$ in $G'$
around $0$, and a map $f:b+B\rightarrow G'$. Assume that
\sn
{\bf (PC1)} \ $\phi: B\rightarrow B'$ is immediate,
\n
{\bf (PC2)} \ for all $y,z\in b+B$,
\[v'(fy-fz-\phi (y-z))\>>\> v'(fy-fz)\>=\>v'\phi (y-z)
\mbox{ \ \ or \ \ } fy-fz=\phi (y-z)=0\;.\]
\sn
Then $f(b+B)\subseteq fb+B'$, and $f:\; b+B\>\rightarrow\> fb+B'$ is
immediate.
\pars
If in addition $\phi$ is injective, then so is $f$, and if\/ $\phi$ is
an embedding of ultrametric spaces with value map $\,\varphi$, then so
is $f$.
%
\end{proposition}
\begin{proof}
%
%
%
Taking $y=z$, we obtain from (PC2) that $\phi(0)=0$. So we can apply
Proposition~\ref{ATC} to find that $\phi$ satisfies (IH1) and (IH2).
Therefore, for $a'\in B' \setminus\{0\}$ we can choose $a\in
\Reg(\phi)\setminus\{0\}$ such that $v'(a'-\phi a)>v'a'$.

Take $y,z\in b+B$ such that $v(y-z)\geq va\,$. By the regularity of $a$,
$v'\phi(y-z) \geq v'\phi a\,$. Hence by (PC2), $v'(fy-fz-\phi(y-z))
=v'\phi (y-z)> v'\phi a$. Now it follows from Proposition~\ref{BCb} that
$f$ is immediate. If in addition, $\phi$ is injective, it follows from
(PC2) that also $f$ is injective. If $\phi$ is an embedding of
ultrametric spaces with value map $\,\varphi$, then $v'\phi (y-z)=
\varphi v(y-z)$ shows that also $f$ is an embedding with value map
$\varphi$.
%
\end{proof}

If the map $\phi$ satisfies the conditions (PC1) and (PC2) of the
foregoing proposition, it will be called a \bfind{pseudo-companion of
$f$ on $b+B$}.

\par\smallskip
We will later need the following fact:
\begin{lemma}                               \label{fftildepc}
Let the situation be as in Proposition~\ref{rangepsd} and let
$\phi,\tilde{\phi}:B\rightarrow B'$ be group homomorphisms. Suppose that
$v'(\tilde{\phi}a-\phi a)\;>\;v'\phi a$ or $\tilde{\phi}a=\phi a=0$ for
all $a\in G$. If $\phi$ is a pseudo companion for $f$ on $b+B$, then so
is $\tilde{\phi}$.
\end{lemma}
\begin{proof}
Assume that $\phi$ is a pseudo-companion of $f$ on $b+B$. Then by
Proposition~\ref{fftilde}, also $\tilde{\phi}$ is immediate. Now take
$y,z\in b+B$. If $\phi(y-z)=0$ then by assumption, $\tilde{\phi}(y-z)=0$.
Otherwise, $v'(fy-fz-\tilde{\phi}(y-z))\geq\min\{v'(fy-fz-\phi(y-z)),
v'(\phi(y-z)-\tilde{\phi} (y-z))\}>v'\phi (y-z)=v'(fy-fz)$. This shows
that also $\tilde{\phi}$ is a pseudo-companion of $f$ on $b+B$.
\end{proof}

%
%
\section{Immediate maps on valued fields and their
finite-dimensional vector spaces}           \label{sectvf}
%
Let $(K,v)$ be a valued field. That is, $v$ is a valuation of its
additive group, $vK$ is a totally ordered abelian group, and the
following additional law holds:
\sn
{\bf (V3)} \ $v(ab)=va+vb$.
\sn
The \bfind{value group} of $(K,v)$ is $vK:=v(K^\times)$. Throughout this
paper, its \bfind{valuation ring} $\{y\in K\mid vy\geq 0\}$ will be
denoted by ${\cal O}$, and its \bfind{valuation ideal} $\{y\in K\mid vy>
0\}$ by ${\cal M}$. The field ${\cal O}/{\cal M}$ is called the
\bfind{residue field} and is denoted by $Kv$. Note that $c{\cal O}=
\{y\in K\mid vy\geq vc\}=B_{vc}(0)$ and $c{\cal O}=\{y\in K\mid vy>vc\}$.

A valued field $(K,v)$ is called \bfind{spherically complete} if the
underlying valued additive group is spherically complete (i.e., if the
underlying ultrametric space is spherically complete).

Main examples for spherically complete fields are the \bfind{power
series fields} $k((G))$ with their \bfind{canonical valuation}. Here,
$k$ can be any field and $G$ any ordered abelian group, and $k((G))$
consists of all formal sums $a=\sum_{g\in G} c_gt^g$ with $c_g
\in k$ and well ordered \bfind{support} $\mbox{supp}(a)=\{g\in G\mid
c_g\ne 0\}$. The canonical valuation on $k((G))$ is given by $va:= \min
\mbox{supp}(a)\in G$ and $v0:=\infty$. Its value group is $G$, and its
residue field is $k$.

An extension $(L,w)\supset (K,v)$ of valued fields is called
\bfind{immediate} if the canonical embedding of $vK$ in $wL$ and the
canonical embedding of $Kv$ in $Lw$ are onto. It is well known that this
holds if and only if as ultrametric spaces, $(K,v)$ is an immediate
subspace of $(L,v)$ (cf.\ [KU4]). A valued field is called
\bfind{maximally valued} if it admits no proper immediate extensions. It
was shown by Krull ([KR]; see also [G]) that for every valued field
$(K,v)$ there is a maximal immediate extension field; this is maximally
valued by definition.

A valued field is maximally valued if and only if it is spherically
complete (cf.\ [P1], [P2], [KU4]). This was essentially proved by
Kaplansky in [KA], using the notion of ``pseudo Cauchy sequence''
instead of ``nest of balls''. Every power series field is spherically
complete (cf.\ [P2], [KU4]). Hence it is maximally valued.

%
%
\subsection{The minimum valuation}
For every $n\in\N$, the valuation $v$ of $K$ induces a valuation of the
$n$-dimensional\linebreak
$K$-vector space $K^n$, called the \bfind{minimum valuation}:
\begin{equation}                            \label{minval}
v(a_1,\ldots,a_n)\>:=\>\min_{1\leq i\leq n} va_i
\end{equation}
for all $(a_1,\ldots,a_n)\in K^n$. This valuation satisfies (V1) and
(V2) for all $a,b\in K^n$, so $(K^n,v)$ is a valued abelian group.
Instead of (V3), it satisfies
\sn
{\bf (V3$'$)} \ $v(ca)=vc+va$ \ for all $c\in K$, $a\in K^n$.

\pars
Again, $u(a,b):=v(a-b)$ makes $K^n$ into an ultrametric space with value
set $vK$. If $0\ne c\in K$, then we write $(c{\cal O})^n$ for the
$n$-fold product $c{\cal O}\times\ldots\times c{\cal O}$ which is the
subgroup of vectors in $K^n$ whose entries all have value $\geq vc\,$;
$(c{\cal M})^n$ is defined similarly. Note that $(c{\cal O})^n=\{ca\mid
a\in {\cal O}^n\}= c{\cal O}^n$ and $(c{\cal M})^n= c{\cal M}^n$. For
$b\in K^n$, $c\in K$,
\[b+c{\cal O}^n\>=\>\{a\in K^n\mid v(a-b)\geq vc\}\>=\>B_{vc}(b)
\mbox{ \ and \ } b+c{\cal M}^n\>=\>\{a\in K^n\mid v(a-b)>vc\}\;.\]

We will say that $(K^n,v)$ is \bfind{spherically complete} if its
underlying ultrametric space $(K^n, u)$ is. Proposition~\ref{prodsphc}
of Section~\ref{sectpd} implies:
\begin{lemma}                               \label{KscKnsc}
If $(K,v)$ is spherically complete, then so is $(K^n,v)$.
\end{lemma}

%
%
\subsection{Pseudo-linear maps}     \label{sectplin}
Take $Y\subseteq K^n$, $0\ne s\in K$ and $f$ a map from $Y$ into $K^n$.
We will say that $f$ is \bfind{pseudo-linear with
pseudo-slope $s$} if for all $y,z\in Y$ such that $y\ne z$,
\begin{equation}                            \label{plm}
v(fy-fz-s(y-z))\>>\>v(fy-fz)\>=\>vs(y-z)\;.
\end{equation}
If $B$ is any ball in $(K^n,v)$ around $0$, then $sB$ is again a ball in
$(K^n,v)$ around $0$ and the map $B\ni x\mapsto sx\in sB$ is an
isomorphism of ultrametric spaces with value map $\,\varphi:\alpha
\mapsto\alpha+vs$. Hence pseudo-linear maps are maps with a particularly
simple pseudo-companion given by multiplication with a suitable scalar.
From Proposition~\ref{rangepsd} we obtain:

\begin{proposition}                               \label{range}
Take $b\in K^n$ and $B$ a ball in $(K^n,v)$ around $0$. Assume that $f:
b+B\rightarrow K^n$ is pseudo-linear with pseudo-slope $s$. Then
$f(b+B)\subseteq fb+sB$, and
\[f:\; b+B\>\rightarrow\> fb+sB\]
is an immediate embedding of ultrametric spaces with
value map $\,\varphi:\alpha\mapsto\alpha+vs$.

\pars
If in addition, $(K,v)$ is spherically complete, then $f$ is an
isomorphism of ultrametric spaces from $b+B$ onto $fb+sB$.
\end{proposition}

%
%
\subsection{Polynomial maps}         \label{sectpm}
Take any $n\in\N$. For any system $f=(f_1,\ldots,f_n)$ of $n$
polynomials in $n$ variables with coefficients in $K$, we denote by
$J_f(b)$ its Jacobian matrix at $b\in K^n$. We will denote by $J^*_f(b)$
the adjoint matrix of $J_f(b)$.

\begin{proposition}                               \label{findpl}
a) \ Take a polynomial $f\in {\cal O}[X]$ and $b\in {\cal O}$ such that
\[s\>:=\>f'(b)\>\ne\> 0\;.\]
Then $f$ induces a pseudo-linear map with pseudo-slope
$s$ from $b+s{\cal M}$ into $f(b)+s^2{\cal M}$.
\mn
b) \ Take $n$ polynomials in $n$ variables $f_1,\ldots,f_n\in {\cal O}
[X_1,\ldots,X_n]$ and $b\in {\cal O}^n$ such that
\[s\>:=\>\det J_f(b)\>\ne\> 0\]
for $f=(f_1,\ldots,f_n)$. If $vs=0$, then $J_f(b)$ is a
pseudo-companion of $f$ on $b+{\cal M}$ and $f$ induces an embedding
from $b+{\cal M}$ into $f(b)+{\cal M}$ with value map
$\varphi=\mbox{\rm id}$.

In the general case, $J^*_f (b)\,f$ induces a pseudo-linear
map with pseudo-slope $s$ from $b+s{\cal M}^n$ into $J^*_f (b) f(b)
+s^2 {\cal M}^n$
\end{proposition}
\begin{proof}
Note that whenever we prove pseudo-linearity, the assertions about the
range of the functions will follow from Proposition~\ref{range}.
\sn
a): \ For a polynomial $f$ in one variable over a field of arbitrary
characteristic, we denote by $f\iT $ its $i$-th formal derivative
(cf.\ [KA], [KU4]). These polynomials are defined such that the
following Taylor expansion holds in arbitrary characteristic:
\begin{equation}                            \label{Tayl1}
f(b+\varepsilon)\>=\>f(b)+\sum_{i=1}^{\deg f}
\varepsilon^i f\iT (b)\;.
\end{equation}
Note that $f'=f^{[1]}$. Since $f\in {\cal O}[X]$, we have that
$f\iT\in {\cal O}[X]$. Since $b\in {\cal O}$, we also have that
$f\iT (b)\in {\cal O}$. Now take $y,z\in b+s{\cal M}$. Write
$y=b+\varepsilon_y$ and $z=b+\varepsilon_z$ with
$\varepsilon_y,\varepsilon_z\in s{\cal M}$. Then by (\ref{Tayl1}),
\begin{equation}                            \label{Tayone}
f(y)-f(z)\>=\>(\varepsilon_y-\varepsilon_z)f'(b)+\sum_{i=2}^{\deg f}
(\varepsilon_y^i-\varepsilon_z^i) f\iT (b)\>=\>
s(y-z) + S(b,\varepsilon_y,\varepsilon_z)\;.
\end{equation}
Since
\[\varepsilon_y^i-\varepsilon_z^i\>=\>(\varepsilon_y-\varepsilon_z)
(\varepsilon_y^{i-1}+(i-1)\varepsilon_y^{i-2}\varepsilon_z+\ldots+
(i-1)\varepsilon_y^{i-2}\varepsilon_z^{i-2}+\varepsilon_y^{i-1})
\in (\varepsilon_y-\varepsilon_z)s{\cal M}\]
for every $i\geq 2$, and since $f\iT (b)\in {\cal O}$, we find that
\[S(b,\varepsilon_y,\varepsilon_z)\in
(\varepsilon_y-\varepsilon_z)s{\cal M}\>=\>s(y-z){\cal M}\;.\]
This proves that
\begin{equation}                            \label{plf1}
v(f(y)-f(z)-s(y-z))\>=\>vS(b,\varepsilon_y,\varepsilon_z)\>>\>
vs(y-z)
\end{equation}
which implies that (\ref{plm}) holds. This proves a).
%
%

\mn
b): \ We write $J=J_f(b)$ and $J^*=J^*_f(b)$. Then $JJ^*=(\det J)
E=sE$ where $E$ is the $n\times n$ identity matrix. Note that $J,J^*\in
{\cal O}^{n\times n}$ by our assumptions on $f$ and $b$. If $y\in K^n$
then we can write $y=cz$ with $c\in K$, $vc=vy$, $z\in {\cal O}^n$ and
$vz=0$. Then $Jy=cJz\in c{\cal O}^n$, hence $vJy=vc+vJz\geq vc=vy$.
Similarly, $vJ^*y\geq vy$ for all $y\in K^n$.

Take $\varepsilon_1, \varepsilon_2\in s{\cal M}^n$. The
multidimensional Taylor expansion gives the following analogue of
(\ref{Tayone}):
\begin{equation}                            \label{Taymult}
f(b+\varepsilon_1)-f(b+\varepsilon_2)\>=\>J(\varepsilon_1-\varepsilon_2)
+ S(b,\varepsilon_1,\varepsilon_2)
\end{equation}
with
\begin{equation}                            \label{Taymultv}
vS(b,\varepsilon_1,\varepsilon_2)\>>\>vs(\varepsilon_1-\varepsilon_2)\;.
\end{equation}
Assume first that $vs=0$. Then also $J^{-1}=\frac{1}{s}J^*\in
{\cal O}^{n\times n}$, so for all $y\in K^n$, $vJ^{-1}y\geq vy$. But
then, $vy=vEy=vJ^{-1}Jy\geq vJy\geq vy$, so equality must hold. We find
that for all $y\in K^n$, $vJy=vy$ and similarly, $vJ^*y=vy$. In
particular, this yields that $J$ induces a value-preserving automorphism
of the valued abelian group $({\cal M}^n,+)$, and an isomorphism of
ultrametric spaces from ${\cal M}^n$ onto ${\cal M}^n$ with value map
$\varphi=\mbox{\rm id}$, with inverse maps induced by $J^{-1}$. From
(\ref{Taymult}) and (\ref{Taymultv}) we obtain that for
$y=b+\varepsilon_1$ and $z= b+\varepsilon_2$ in $b+{\cal M}$,
\[
v(f(y)-f(z)\>-\>J(y-z))\> >\>vs(y-z)\>=\>v(y-z)\>=\>vJ(y-z)\;.
\]
This proves that $J$ is a pseudo-companion of $f$ on $b+{\cal M}$. From
Proposition~\ref{rangepsd} we infer that $f$ induces an embedding from
$b+{\cal M}$ into $f(b)+J{\cal M}=f(b)+{\cal M}$ with value map
$\varphi=\mbox{\rm id}$.

\parm
Now we turn to the general case. We compute:
\begin{eqnarray*}
J^*f(y)-J^*f(z) & = & J^*(f(b+y-b)-f(b+z-b))\\
 & = & J^*J(y-z)\,+\,J^*S(b,y-b,z-b)\\
 & = & s(y-z) \,+\,J^*S(b,y-b,z-b)\;.
\end{eqnarray*}
By (\ref{Taymultv}),
\[vJ^*S(b,y-b,z-b) \>\geq\> vS(b,y-b,z-b) \> >\> vs(y-z)\;.\]
Hence,
\[v\left(J^*f(y)-J^*f(z)-s(y-z)\right)\>=\>
vJ^*S(b,y-b,z-b)\> >\>vs(y-z)\;.\]
This proves our assertion for the map $J^*_f (b)\,f$.
%
%
%
\end{proof}

Note that in the one-dimensional case ($n=1$), we may write $\det J_f(b)
= f'(b)$ and $J^*_f(b)=1$; in this way, the definition of $f_{\langle
b\rangle}$ in the one-dimensional case becomes a special case of the
definition for the multi-dimensional case.

If $vs>0$ in the multi-dimensional case, then in general $J_f(b)$ will
not be a pseudo-companion of $f$. It is necessary to transform $f$ in
order to obtain suitable pseudo-companions. We have shown above that
this can be done so that one even obtains pseudo-linear functions.

From Proposition~\ref{findpl} together with Propositions~\ref{range}
and~\ref{rangepsd}, we obtain:

\begin{theorem}                             \label{MT2}
Assume that $(K,v)$ is spherically complete.\sn
a) \ Take a polynomial $f\in {\cal O}[X]$ and $b\in {\cal O}$ such
that $s:=f'(b)\ne 0$. Then $f$ induces a pseudo-linear isomorphism of
ultrametric spaces from $b+s{\cal M}$ onto $f(b)+s^2{\cal M}$, with
pseudo-slope $s$.
\sn
b) \ Take $n$ polynomials in $n$ variables $f_1,\ldots,f_n\in {\cal O}
[X_1,\ldots,X_n]$ and $b\in {\cal O}^n$ such that $s:=\det J_f(b)\ne 0$
for $f= (f_1,\ldots,f_n)$. If $vs=0$, then $f$ induces an embedding of
ultrametric spaces from $b+{\cal M}$ onto $f(b)+{\cal M}$.

In the general case, $J^*_f (b)\,f$ induces a pseudo-linear isomorphism
of ultrametric spaces from $b+s{\cal M}^n$ onto $J^*_f (b)\,f(b)+ s^2
{\cal M}^n$, with pseudo-slope $s$.
%
%
\end{theorem}

%
%
\subsection{Hensel's Lemma and Implicit Function Theorem revisited}
\label{sectHLr}
Let us apply Theorem~\ref{MT2} to prove that Hensel's Lemma
holds for every spherically complete valued field $(K,v)$. We
prove the following version of Hensel's Lemma, which is often called
``Newton's Lemma'':
\begin{theorem}                             \label{HL}
Let $(K,v)$ be a spherically complete valued field.
Then $(K,v)$ satisfies the one-dimensional Newton's Lemma:\n
Take $f\in {\cal O}[X]$ and assume that $b\in {\cal O}$ is such that
$vf(b)>2vf'(b)$. Then there exists a unique root $a$ of $f$ such that
$v(a-b)=vf(b)-vf'(b)>vf'(b)$.
\end{theorem}
\begin{proof}
The inequality $vf(b)>2vf'(b)$ implies that $s:=f'(b)\ne 0$. Hence by
Theorem~\ref{MT2}, $f$ induces a pseudo-linear isomorphism of
ultrametric spaces from $b+s{\cal M}$ onto $f(b)+s^2{\cal M}$, with
pseudo-slope $s$. Since $vf(b)>2vf'(b)=vs^2$, we have that $f(b)\in
s^2{\cal M}$, that is, $f(b)+s^2{\cal M}=s^2{\cal M}$. Therefore, $0\in
f(b)+s^2{\cal M}$. Since $f$ induces a bijection from $b+s{\cal M}$ onto
$f(b)+s^2 {\cal M}$, there is a unique $a\in b+s{\cal M}$ such that
$f(a)=0$. We have that $v(a-b)=v(f(a)-f(b))-vf'(b)=vf(b)-vf'(b)>vf'(b)$.
\end{proof}

Here is the multi-dimensional version:
\begin{theorem}                             \label{sphcomHL}
Let $(K,v)$ be a spherically complete valued field. Then $(K,v)$
satisfies the multi-dimensional Newton's Lemma:\n
Let $f= (f_1,\ldots,f_n)$ be a system of $n$ polynomials in $n$
variables with coefficients in ${\cal O}$. Assume that $b\in
{\cal O}^{\,n}$ is such that $vf(b)>2v\det J_f(b)$. Then there
exists a unique $a\in {\cal O}^{\,n}$ such that $f(a)=0$ and
$v(a-b)= vJ^*_f(b)f(b)-v\det J_f(b) > v\det J_f(b)$.
\end{theorem}
\begin{proof}
The inequality $vf(b)>2v\det J_f(b)$ implies that $s:=\det J_f(b)\ne 0$.
Hence by Theorem~\ref{MT2}, $J^*f$ induces an isomorphism of
ultrametric spaces from $b+s{\cal M}^n$ into $J^* f(b) +s^2
{\cal M}^n$, where $J^*=J^*_f(b)$. Since $vf(b)>vs^2$, we have that
$f(b) \in s^2{\cal M}^n$ and hence also $J^* f(b)\in s^2{\cal M}^n$
(since $J^*\in {\cal O}^{n\times n}$).
That is, $J^*f(b)+s^2{\cal M}^n= s^2{\cal M}^n$. Therefore, $0\in
J^*f(b)+ s^2{\cal M}^n$. Since $J^*f$ induces a bijection from
$b+s{\cal M}^n$ onto $J^*s^{-2}f(b)+ {\cal M}^n$, there is a unique
$a\in b+s{\cal M}^n$ such that $J^*f(a)=0$. Since $J^*$ is
invertible, we have that $f(a)=0\Leftrightarrow J^*f(a)=0$. Hence,
$a$ is the unique element in $b+s{\cal M}^n$ such that $f(a)=0$.
We have that $v(a-b)=v\left(J^*_f(b)f(a)-J^*_f(b)f(b)\right)
-v\det J_f(b) =vJ^*_f(b)f(b)-v\det J_f(b)> v\det J_f(b)$.
\end{proof}

Note that like in the one-dimensional case, also in the multi-dimensional
case the proof of Newton's Lemma can be reduced by transformation to a
simpler case where we would in fact obtain the identity as a
pseudo-companion. But as we have already shown that even in the
general case we can derive suitable pseudo-linear maps from $f$, it is
much easier to employ them directly in the proof of the
multidimensional Newton's Lemma.

\pars
A valued field $(K,v)$ is called \bfind{henselian} if the extension of
$v$ to the algebraic closure $\tilde{K}$ of $K$ is unique. It is well
known that this holds if and only if $(K,v)$ satisfies the
one-dimensional Newton's Lemma (see, e.g., [KU4]). We are now going to
show that the multi-dimensional Newton's Lemma holds in every henselian
field.

\begin{theorem}                             \label{multhens}
A valued field $(K,v)$ is henselian if and only if it
satisfies the multidimensional Newton's Lemma.
\end{theorem}
\begin{proof}
$\Rightarrow$:\ \ Let $(K,v)$ be henselian. Take $(L,v)$
to be a maximal immediate extension of $(K,v)$. Then $(L,v)$ is
spherically complete. By the foregoing theorem, $(L,v)$ satisfies the
multidimensional Newton's Lemma. Denote by ${\cal O}$ the valuation
ring of $K$, and by ${\cal O}_L$ that of $L$. Now assume that the
hypothesis of the multidimensional Newton's Lemma is satisfied by a
system $f$ of polynomials with coefficients in ${\cal O}$ and by $b\in
{\cal O}^{\,n}$. It follows that there is a unique $a=(a_1,\ldots,a_n)
\in {\cal O}_L ^{\,n}$ such that $f(a)=0$ and $v(a -b)>v\det J_f (b)$.
From the latter, it follows that $v\det J_f (a)= v\det J_f (b)$ and in
particular, $\det J_f(a) \not=0$. Now [L], Chapter X, \S7, Proposition
8, shows that the elements $a_1,\ldots,a_n$ are separable algebraic over
$K$. On the other hand, for every $\sigma\in\mbox{Aut}\,(\tilde{K}|K)$,
the element $\sigma a=(\sigma a_1,\ldots,\sigma a_n)$ satisfies
$f(\sigma a)=\sigma f (a)=0$ and $v(\sigma a-b) =\min_i v(\sigma a_i-
b_i)=\min_i v\sigma (a_i-b_i)=\min_i v(a_i-b_i)= v(a -b)>v\det J_f (b)$
(note that $v\sigma=v$ because $(K,v)$ is henselian). By the uniqueness
of $a$, it follows that $\sigma a=a$ for every $\sigma\in\mbox{Aut}\,
(\tilde{K}|K)$, that is, $a\in K^n$, as required.
\mn
$\Leftarrow$:\ \ If $n=1$, then $\det J_f(b)= f'_1(b_1)$, and the
assertion is precisely the assertion of the one-dimensional Newton's
Lemma. Hence the multidimensional Newton's Lemma implies that $(K,v)$ is
henselian.
\end{proof}

Using the multidimensional Newton's Lemma, one can prove the
multidimensional\linebreak
\bfind{Implicit Function Theorem}:
\begin{theorem}
Let $(K,v)$ be a henselian field, and let $f_1,\ldots,f_n\in
{\cal O}[X_1,\ldots,X_m, Y_1,\ldots,Y_n]$ with $m<n$. Set
$Z=(X_1,\ldots,X_m, Y_1,\ldots,Y_n)$ and
\[J(Z):=\left(
\begin{array}{ccc}
\frac{\partial f_1}{\partial Y_1}(Z) &\ldots&
\frac{\partial f_1}{\partial Y_n}(Z)\\
\vdots & & \vdots\\
\frac{\partial f_m}{\partial Y_1}(Z) &\ldots&
\frac{\partial f_m}{\partial Y_n}(Z)
\end{array}\right)\;\;.\]
Assume that $f_1,\ldots,f_n$ admit a common zero $z= (x_1,\ldots,x_m,
y_1,\ldots,y_n)\in {\cal O}^{m+n}$ and that the determinant of $J(z)$ is
nonzero. Then for all $(x'_1,\ldots,x'_m)\in {\cal O}^m$ with
$v(x_i-x'_i) >2v\det J(z)$, $1\leq i\leq m$, there exists a unique tuple
$(y'_1,\ldots, y'_n)\in {\cal O}^n$ such that $(x'_1,\ldots,x'_m,y'_1,
\ldots,y'_n)$ is a common zero of $f_1,\ldots,f_m\,$, and
\[\min_{1\leq i\leq n} v(y_i-y'_i)\geq \min_{1\leq i\leq m}
v(x_i-x'_i)-v\det J(z)\;.\]
\end{theorem}
\begin{proof}
We observe that the entries of $J(Z)$ and its adjoint matrix $J^*(Z)$
are polynomials in $X_1,\ldots,X_m,Y_1,\ldots,Y_n$ with coefficients in
${\cal O}$. We set $b= (x'_1,\ldots,x'_m,y_1,\ldots,y_n)$. Then $J^*(b)$
is the adjoint matrix for $J(b)$, and the entries of both matrices lie
in ${\cal O}$. In particular, this implies that $vJ^*(b)f(b) \geq
vf(b)$.

By assumption, $f_i(z)=0$ for $1\leq i\leq m$. Hence, the condition
$v(x_i-x'_i)>2\det vJ(a)$, $1\leq i\leq m$, will imply that
\begin{eqnarray*}
v f_i(b) & = & v\left(f_i(x'_1,\ldots,x'_m,y_1,\ldots,y_n)
-f(x_1\ldots,x_m,y_1,\ldots,y_n)\right)\> \geq \>
\min_{1\leq i\leq m} v(x_i-x'_i)\\
& > & 2v\det J(x_1\ldots,x_m,y_1,\ldots,y_n) \>=\>
2v\det J(x'_1,\ldots,x'_m,y_1,\ldots,y_n) \>=\> 2v\det J(b)
\end{eqnarray*}
for $1\leq i\leq m$. In particular, $\det J(b)\ne 0$. Hence by
Theorem~\ref{multhens},
%
%
there is a unique common
zero $(y'_1,\ldots,y'_n)\in {\cal O}^n$ of the polynomials
$f_i(x'_1,\ldots, x'_m,Y_1,\ldots,Y_n)$, $1\leq i\leq n$, such that
\begin{eqnarray*}
\min_{1\leq i\leq n} v(y_i-y'_i) & \geq & vJ^*(b)f(b)-v\det J(b)
\>=\>vJ^*(b)f(b)-v\det J(z)\\
& \geq & \min_{1\leq i\leq m} vf_i(b) -v\det J(z) \>\geq\>
 \min_{1\leq i\leq m} v(x_i-x'_i)-v\det J(z)
\end{eqnarray*}
This proves our assertion.
\end{proof}

%
%
\subsection{An infinite-dimensional Implicit Function Theorem}
                                            \label{sectidIFT}
From our result in Section~\ref{sectpd} it follows that an infinite
power $Y^I$ of an ultrametric space $Y$ can be equipped with an
ultrametric $u^I$ (analogous to the minimum valuation) if the value set
$uY$ is well ordered. In this case, if $(Y,u)$ is spherically complete,
then so is $(Y^I,u^I)$. So we obtain the following corollary to our Main
Theorem~\ref{MT} and to Proposition~\ref{rangepsd}:
\begin{corollary}                           \label{rangecor}
a) \ Take two ultrametric spaces $(Y,u)$ and $(Y',u')$, and an arbitrary
index set $I$. Assume that $uY$ is well ordered, $f:\;Y^I \rightarrow
Y'$ is immediate and that $(Y,u)$ is spherically complete. Then $f$ is
surjective and $(Y',u')$ is spherically complete.
\sn
b) \ Take two valued abelian groups $(G,v)$ and $(G',v')$, and an
arbitrary index set $I$. Assume that $vG$ is well ordered, $b\in G^I$,
$B$ is a ball around $0$ in $G^I$, $f:\;G^I \rightarrow G'$ has a
pseudo-companion on $b+B$, and that $(G,v)$ is spherically
complete. Then $f$ is surjective and $(G',v')$ is spherically complete.
\end{corollary}

In the case of a valued field $(K,v)$ we cannot do the same since if the
valuation is non-trivial, the value group will not be well ordered.
If the valuation is not discrete (i.e., its value group is not
isomorphic to $\Z$), then not even the value set $v{\cal O}:=v({\cal
O}\setminus \{0\})$ of the valuation ring is well ordered. But we may be
interested in infinite systems of polynomials with coefficients in a
subring $R$ of ${\cal O}$ with well ordered value set $vR:=v(R\setminus
\{0\})$. We set ${\cal M}_R:= \{a\in R\mid va>0\}$.

Note that $(R,v)$ is not necessarily spherically complete, even if
$(K,v)$ is. So we will assume that $(R,v)$ is spherically complete.

We generalize the definitions of \bfind{minimum valuation} and of
\bfind{pseudo linear map} in the obvious way. If $a=(a_i)_{i\in I}\in
R^I$, then $va:=\min_{i\in I} va_i\,$. If $Y\subseteq R^I$, $0\ne s\in
R$ and $f$ a map from $Y$ into $R^I$, then $f$ is pseudo-linear with
pseudo-slope $s$ if (\ref{plm}) holds for all $y,z\in Y$ such that $y\ne
z$. We then have the following application of Proposition~\ref{rangepsd}
together with Proposition~\ref{prodsphc}:

\begin{proposition}                               \label{rangeR}
Take $b\in R^I$ and $B$ a ball in $(R^I,v)$ around $0$. Assume that $f:
b+B\rightarrow R^I$ is pseudo-linear with pseudo-slope $s\in R$ and that
$(R,v)$ is spherically complete. Then $f$ is an isomorphism of
ultrametric spaces from $b+B$ onto $fb+sB$.
\end{proposition}

If the map is given by an infinite system of polynomials $f=
(f_k)_{k\in I}$ in infinitely many variables $X_i\,$, $i\in I$, and with
coefficients in $R$, then we may consider the infinite matrix $J_f(b)\in
R^{I\times I}$. Note that this matrix has only finitely many non-zero
entries in every row. We denote by $R^{(I\times I)}$ all matrices in
$R^{I\times I}$ which have only finitely many non-zero entries in every
row and every column. If every variable appears only in finitely many
$f_k$, then $J_f(b)\in R^{(I\times I)}$.

If we assume that $R$ is spherically complete, we can consider a larger
class of matrices. We denote by $R^{((I\times I))}$ all matrices in
$R^{I\times I}$ which for each $\alpha\in vR$ have only finitely many
entries of value $\leq\alpha$ in every row and every column. For every
two matrices in $R^{((I\times I))}$, their product can be computed and
lies again in $R^{((I\times I))}$. It is possible that $J_f(b)\in
R^{((I\times I))}$ even when there are variables that appear in
infinitely many $f_k$.

We define ${\cal M}_R^{(I\times I)}$ and ${\cal M}_R^{((I\times I))}$
analogously and note that $R^{(I\times I)}$, $R^{((I\times I))}$,
${\cal M}_R^{(I\times I)}$ and ${\cal M}_R^{((I\times I))}$ are all
closed under matrix addition and multiplication and under scalar
multiplication. Further, $R^{(I\times I)}{\cal M}_R^{(I\times I)}
\subseteq {\cal M}_R^{(I\times I)}$, ${\cal M}_R^{(I\times I)}
R^{(I\times I)}\subseteq {\cal M}_R^{(I\times I)}$, $R^{((I\times I))}
{\cal M}_R^{((I\times I))}\subseteq {\cal M}_R^{((I\times I))}$ and
${\cal M}_R^{((I\times I))}R^{((I\times I))}\subseteq {\cal
M}_R^{((I\times I))}$.

\parm
We are not able to use determinants here. Still, we can use our original
approach if $J_f(b)$ has an inverse. But we can even work with less than
invertibility. Given matrices $M,M^\circ$ in $R^{(I\times I)}$, or in
$R^{((I\times I))}$ if $R$ is spherically complete, we will say that
$M^\circ$ is a \bfind{pseudo-inverse of $M$} if the matrices
$MM^\circ-E$ and $M^\circ M-E$ are in ${\cal M}_R^{I\times I}$,
where $E$ denotes the $I\times I$-identity matrix.

Actually, we also do not need that the ring $R$ is a subring of a valued
field. It suffices to assume that it is a valued abelian group with its
multiplication satisfying (V3), and that its value set is a well ordered
subset of an ordered abelian group. It then follows that the value set
does not contain negative elements. In particular, all entries of $M\in
R^{I\times I}$ have value $\geq 0$. This implies that $vMa\geq va$ for
all $a\in R^I$. Since $vR$ is well ordered, it contains a minimal
positive value $\alpha_0$. If $M$ is in ${\cal M}_R^{(I\times I)}$ or in
${\cal M}_R^{((I\times I))}$, then all entries of $M$ have value $\geq
\alpha_0$. It then follows that $vMa\geq va+\alpha_0>va$ for all $a\in
R^I$.

\begin{lemma}                               \label{pi}
Take $M,M^\circ$ in $R^{(I\times I)}$, or in $R^{((I\times I))}$ if $R$
is spherically complete. Assume that $M^\circ$ is a pseudo-inverse of
$M$. Then the following holds:
\n
1) \ For all $a\in R^I$, $vMa=va$ and $vM^\circ a=va$; in particular,
$M,M^\circ\notin {\cal M}_R^{I\times I}$ and the value set $vR$ must
contain $0$.
\n
2) \ If $M'$ is in $R^{(I\times I)}$, or in $R^{((I\times I))}$
respectively, such that $M'-M\in {\cal M}_R^{I\times I}$, then
$M^\circ$ is also a pseudo-inverse of $M'$.
\n
3) \ Both $M$ and $M^\circ$ induce immediate embeddings of the
ultrametric space $R^{I}$ in itself with value map id, and the same
holds on every ball around 0 in $R^I$.
\end{lemma}
\begin{proof}
1): \ For all $a\in R^I$ we have that $v(MM^\circ a\,-\,a)=v((MM^\circ
-E)a>va$ and hence $va=vMM^\circ a\geq vM^\circ a\geq va$. It
follows that equality holds everywhere, which gives $vM^\circ a=va$.
Interchanging $M$ and $M^\circ$, we obtain $vMa=va$.
\pars
2): \ We compute: $M'M^\circ-E=(M'-M)M^\circ+MM^\circ-E\in
{\cal M}_R^{I\times I}$, and similarly for $M^\circ M'-E$.
\pars
3): \ It suffices to show that for every ball $B$ around 0 in $R^I$, $M$
induces an immediate embedding of $B$ into itself with value map id.
Since $vMa=va$ for all $a\in R^I$, we have $MB\subseteq B$ and that $M$
induces an injective map on $B$ with value map id. As $M$ induces a
group homomorphism, we only have to show now that for every $a'\in B
\setminus \{0\}$ there is $a\in B$ such that (IH1) and (IH2) of
Proposition~\ref{ATC} hold for $M$ in the place of $f$. As $vM^\circ
a'=va'$, we have that $a:=M^\circ a'\in B$. Further, $v(a'-Ma)=
v(a'-MM^\circ a')=v(E-MM^\circ)a'>va'$. Finally, if $b\in B$ with
$va\leq vb$, then $vMa=va\leq vb=vMb$.
\end{proof}

\begin{proposition}                         \label{infpc}
Assume that $(R,v)$ is spherically complete. Take any index set $I$ and
a system of polynomials $f= (f_k)_{k\in I}$ in variables $Y_i\,$, $i\in
I$, with coefficients in $R$. Take $b\in R^I$ and suppose that $J_f(b)$
lies in $R^{((I\times I))}$ and admits a pseudo-inverse in $R^{((I\times
I))}$. Then $J_f(b)$ is a pseudo-companion of $f$ on $b+ {\cal M}_R^I$,
and $f$ is an isomorphism from $b+{\cal M} _R^I$ onto $f(b)+{\cal
M}_R^I$ with value map id. The system $f$ has a zero on $b+{\cal M}_R^I$
(which then is unique) if and only if $vf(b)>0$.
\end{proposition}
\begin{proof}
Since $J=J_f(b)$ has a pseudo-inverse, we know from the previous lemma
that $J$ induces an immediate embedding of ${\cal M}_R^I$ in itself
with value map id.

Take $\varepsilon_1, \varepsilon_2\in {\cal M}_R^I$. An
infinite-dimensional version of the multidimensional Taylor expansion
gives the infinite-dimensional analogue of (\ref{Taymult}) and
(\ref{Taymultv}), with $s=1$. We obtain that for $y=b+\varepsilon_1$
and $z= b+\varepsilon_2$ in $b+{\cal M}_R^I$ with $y\ne z$,
\[
v(f(y)-f(z)\>-\>J(y-z))\> >\>v(y-z)\>=\>vJ(y-z)\;.
\]
This proves that $J$ is a pseudo-companion of $f$ on $b+{\cal M}_R^I$. From
Proposition~\ref{rangepsd} we infer that $f$ induces an embedding of
$b+{\cal M}_R^I$ in $f(b)+J{\cal M}_R^I\subseteq f(b)+{\cal M}_R^I$
with value map $\varphi=\mbox{\rm id}$.

The remaining assertions now follow from Proposition~\ref{rangepsd} and
Theorem~\ref{MT}.
\end{proof}

Now we can prove an
\bfind{infinite-dimensional Implicit Function Theorem}:
\begin{theorem}
Take any index sets $I$ and $I'$ and a system of polynomials $f=
(f_k)_{k\in I}$ in variables $X_j\,$, $j\in I'$, and $Y_i\,$, $i\in I$,
with coefficients in $R$, and such that each variable $Y_i$ appears in
only finitely many $f_k$. Assume that $(R,v)$ is spherically complete.
Set $Z=(X_j,Y_i\mid j\in I',i\in I)$ and
\[J(Z):=\left(\frac{\partial f_k}{\partial Y_i}(Z)\right)_
{k,i\in I}\;\;.\]
Assume that the polynomials $f_k$, $k\in I$, admit a common zero $z=
(x_j,y_i\mid j\in I',i\in I)$ in $R^{I'\cup I}$ such that $J(z)$ admits
a pseudo-inverse in $R^{((I\times I))}$. Then for all $(x'_j)_{j\in I'}
\in R^{I'}$ with $v(x_j-x'_j)>0$ there exists a unique $(y'_i)_{i\in I}
\in R^I$ such that $z'= (x'_j,y'_i\mid j\in I',i\in I)$ is a common zero
of the polynomials $f_k\,$, $k\in I$,
and
\[\min_{i\in I} v(y_i-y'_i)\geq \min_{j\in I'} v(x_j-x'_j)\;.\]
\end{theorem}
\begin{proof}
We set $\tilde{z}:=(x'_j,y_i\mid j\in I',i\in I)$ and observe that our
condition that $v(x_j-x'_j)>0$ implies that $v\left(\frac{\partial f_k}
{\partial Y_i}(\tilde{z})-\frac{\partial f_k} {\partial Y_i}(z)\right)
>0$. From part 2) of Lemma~\ref{pi} it thus follows that the
pseudo-inverse of $J(z)$ is also a pseudo inverse of $J(\tilde{z})$.
(Note that $J(z),J(\tilde{z})\in R^{(I\times I)}$ by our condition on
the variables $Y_i$.)

For each $k\in I$ we set $g_k(Y_i\mid j\in I):=f_k(x'_j,Y_i\mid j\in I',
i\in I)$. Further, we set $b:=(y_i\mid i\in I)$. We consider the system
$g=(g_k)_{k\in I}$. From Proposition~\ref{infpc} we infer that
$J_g(b)=J(\tilde{z})$ is a pseudo-companion of $g$ on $b+ {\cal M}_R^I$.
By assumption, $f_k(z)=0$ for $k\in I$. Hence, the condition
$v(x_j-x'_j)>0$ will imply that
\[vg_k(b)\>=\>v f_k(\tilde{z}) \>=\> v(f_k(\tilde{z})
-f_k(z))\>\geq\>\min_{j\in I'} v(x_j-x'_j)>0\;.\]
Hence $vg(b)>0$ and by Proposition~\ref{infpc} the system $g$ has a
unique zero $a=(y'_i\mid i\in I)$ on $b+ {\cal M}_R^I$. It satisfies
\[\min_{i\in I} v(y_i-y'_i)\>=\>v(b-a)\>=\>v(g(b)-g(a))\>=\>vg(b)
\>\geq\>\min_{j\in I'} v(x_j-x'_j)\;.\]
\end{proof}

\begin{remark}
In our theorem we needed the assumption on the variables $Y_i$ in order
to have only finitely many non-zero polynomials in each row and each
column of $J(Z)$. Without this it is not automatic that the conditions
$J(z)\in R^{((I\times I))}$ and $v(x_j-x'_j)>0$ imply that $J(\tilde{z})
\in R^{((I\times I))}$. We can drop the condition on the variables if we
assume instead that $J(\tilde{z}) \in R^{((I\times I))}$ and that it has
a pseudo-inverse in $R^{((I\times I))}$.
\end{remark}

%
%
\subsection{Power series maps on valuation ideals}  \label{sectpsf}
Take any field $k$ and any ordered abelian group $G$. We endow $k((G))$
with the canonical valuation $v$ and denote the valuation ideal by
${\cal M}$. Every power series
\begin{equation}                            \label{ps}
f(X)=\sum_{i\in\N}^{}c_iX^i\in k[[X]]
\end{equation}
defines in a canonical way a map $f:{\cal M}\rightarrow {\cal M}$
(note: $0\notin\N$ in our notation).
This can be shown by use of Neumann's Lemma, cf.\ [DMM1]. We note that
for every integer $r>1$ and every $y,z\in {\cal M}$,
\begin{equation}                            \label{va-br}
v(y^r-z^r) \>>\> v(y-z)\;.
\end{equation}
Therefore, if $c_1\ne 0$, we have that
\begin{equation}                            \label{ps-}
v(f(y)- f(z)-c_1(y-z))\>=\>v\sum_{i\geq 2}c_i(y^i-z^i)\>>\>v(y-z)
\>=\>vc_1(y-z)
\end{equation}
because $vc_i=0$ for all $i$. So we see that $f$ is pseudo-linear
with slope $c_1$ if $c_1\ne 0$. By Proposition~\ref{range}, we obtain:
\begin{theorem}
If $f:{\cal M}\rightarrow {\cal M}$ is defined by the power series
(\ref{ps}), then $f$ is an isomorphism of ultrametric spaces.
\end{theorem}

A similar result holds for power series with generalized exponents
(which for instance are discussed in [DS]). Take any subgroup $G$ of
$\R$ and a generalized power series of the form
\begin{equation}                            \label{gps}
f(X)=\sum_{i\in\N}^{}c_iX^{r_i}\in k[[X^G]]
\end{equation}
where $r_i\,$, $i\in\N$, is an increasing sequence of positive real
numbers in $G$. Suppose that the power functions $y\mapsto y^{r_i}$ are
defined on ${\cal M}$ for all $i$. Then again, the generalized power
series (\ref{gps}) defines a map $f:{\cal M}\rightarrow {\cal M}$. We
note that (\ref{va-br}) also holds for every real number $r>1$ for which
$y \mapsto y^r$ is defined on ${\cal M}$. Hence if $c_1\ne 0$ and
$r_1=1$, then (\ref{ps-}) holds, with the exponent $i$ replaced by
$r_i\,$.
%
%
This shows again that $f$ is pseudo-linear with pseudo-slope $c_1\,$.
If, however, $r_1\ne 1$, we may think of writing
$f(y)=\tilde{f}(y^{r_1})$ with
\[\tilde{f}(X)\>=\>\sum_{i\in\N}^{}c_iX^{r_i/r_1}\;.\]
If the power functions $y\mapsto y^{r_i/r_1}$ are defined on ${\cal M}$
for all $i$, then $\tilde{f}$ defines a pseudo-linear map from
${\cal M}$ to ${\cal M}$ with pseudo-slope $c_1\,$. So we obtain:
\begin{theorem}
Suppose that the power functions $y\mapsto y^{r_i}$ and $y\mapsto
y^{r_i/r_1}$ are defined on ${\cal M}$ for all $i$, and that $y\mapsto
y^{r_1}$ is surjective. If $f:{\cal M}\rightarrow {\cal M}$ is defined
by the power series (\ref{gps}) with $c_1\ne 0$, then $f$ is surjective.
\end{theorem}

%
%
\subsection{Power series maps and infinite-dimensional Implicit
Function Theorems}                              \label{sectpsiift}
We use again the notations and assumptions from Section~\ref{sectidIFT}.
We take $R[[X_j,Y_i\mid j\in I',i\in I]]$ to be the set of all formal
power series in the variables $X_j,Y_i$ in which for every $n\in\N$ only
finitely many of the $X_j,Y_i$ appear to a power less than $n$. In the
previous section, our power series had well defined values because we
were operating in a power series field $k((G))$. Here, we will assume
throughout that $R$ is spherically complete. But this alone does not
a priori give us well defined values of the power series on
${\cal M}_R^{I'\cup I}$. So we will assume that we have some canonical
way to determine the value of a given power series at an element of
${\cal M}_R^I$. This holds for instance if $vR$ is archimedean, i.e., is
a subsemigroup of an archimedean ordered abelian group.

To every power series $g\in R[[Y_i\mid i\in I]]$ we associate its
\bfind{0-linear part} $L_g^0$, by which we mean the sum of all of its
monomials of total degree 1 and with a coefficient in $R$ of value 0.
This is a polynomial, i.e., contains only finitely many of the variables
$Y_i$. We set $Y=(Y_i\mid i\in I)$.

\begin{theorem}                             \label{IIFTPS}
Assume that $(R,v)$ is spherically complete. Take any index sets $I$ and
$I'$ and a system $f= (f_k)_{k\in I}$ where $f_k\in R[[X_j,Y_i\mid j\in
I',i\in I]]$. Assume that $f_k$, $k\in I$, admit a common zero
$z= (x,y)$, $x\in {\cal M}_R^{I'}$, $y\in {\cal M}_R^I$, such that for
the map $L(Y)=L^0_{f(x,Y)}(Y): {\cal M}_R^I \rightarrow {\cal M}_R^I$
the following holds: for every $a'\in {\cal M}_R^I\setminus\{0\}$ there
is some $a\in {\cal M}_R^I$ such that
\[v(a'-La)\> >\>va'\quad\mbox{ and }\quad va\>=\>va'\;.\]
Take $x'=(x'_j)_{j\in I'}\in {\cal M}_R^{I'}$, set $\alpha=v(x-x')$ and
$g(Y)=f(x',Y)$ and suppose that
\begin{equation}                            \label{IIFTPSc}
v(gw-gw'-L(w-w'))\>>\> v(gw-gw')\;\mbox{ for all distinct }
w,w'\in B_\alpha(y)\>.
\end{equation}
Then there exists a unique $(y'_i)_{i\in I}\in {\cal M}_R^I$ such that
$z'= (x'_j,y'_i\mid j\in I',i\in I)$ is a common zero of
$f_k\,$, $k\in I$, and
\[\min_{i\in I} v(y_i-y'_i)\geq \alpha\;.\]
\end{theorem}
\begin{proof}
Note that $L_{f(x',Y)}(Y)=L_{f(x,Y)}(Y)=L(Y)$. We claim that $L$ is a
pseudo-companion of $f(x',Y):{\cal M}_R^I \rightarrow {\cal M}_R^I$ on
$B_\alpha(y)$. Condition (PC2) holds by assumption. As $L$ is a group
homomorphism, our conditions together with Proposition~\ref{ATC} show
that $L:{\cal M}_R^I \rightarrow {\cal M}_R ^I$ is immediate; note that
(IH2) holds because if $va\leq vb$ then $vLa= va\leq vb\leq vLb$. Now
the assertion of our theorem follows as in earlier proofs.
%
\end{proof}

The following version of the above theorem has a similar proof:
\begin{theorem}
Assume that $(R,v)$ is spherically complete.
Take any index sets $I$ and $I'$ and a system $f= (f_k)_{k\in I}$ where
$f_k\in R[X_j\mid j\in I'][[X_j\mid i\in I]]$.
Assume that $f_k$, $k\in I$, admit a common zero $z= (x,y)$, $x\in
R^{I'}$, $y\in {\cal M}_R^I$, such that $L(Y)=L^0_{f(x,Y)}(Y)$ satisfies
the same condition as in Theorem~\ref{IIFTPS}. Take $x'=(x'_j)_{j\in I'}
\in R^{I'}$ such that $\alpha=v(x-x') >0$. Suppose that
(\ref{IIFTPSc}) holds for $g(Y)=f(x',Y)$. Then there exists a unique
$(y'_i)_{i\in I}\in {\cal M}_R^I$ such that $z'= (x'_j,y'_i\mid j\in
I',i\in I)$ is a common zero of the polynomials $f_k\,$, $k\in I$,
and $\min_{i\in I} v(y_i-y'_i)\geq \alpha$.
\end{theorem}

\parm
Alternatively, in order to obtain maps on all of $R$, one can consider
convergent power series. We let $R\{\{X_j,Y_i\mid j\in I',i\in I\}\}$ be
the set of all formal power series in the variables $X_j,Y_i$ in which
for every $\alpha\in vR$ only finitely many monomials have coefficients
of value less than $\alpha$. Again we assume that $R$ is spherically
complete. Then every convergent power series defines a map from $R$ into
$R$. In a similar way as before, one can prove:

\begin{theorem}
Assume that $(R,v)$ is spherically complete. Take any index sets $I$ and
$I'$ and a system $f= (f_k)_{k\in I}$ where $f_k\in R\{\{X_j,Y_i\mid
j\in I',i\in I\}\}$. Assume that $f_k$, $k\in I$, admit a common zero
$z= (x,y)$, $x\in R^{I'}$, $y\in R^I$, such that $L(Y)=L^0_{f(x,Y)}(Y)$
satisfies the same condition as in Theorem~\ref{IIFTPS}. Take $x'=
(x'_j)_{j\in I'}\in R^{I'}$ such that $\alpha=v(x-x') >0$. Suppose that
(\ref{IIFTPSc}) holds for $g(Y)=f(x',Y)$. Then there exists a unique
$(y'_i)_{i\in I}\in R^I$ such that $z'= (x'_j,y'_i\mid j\in I',i\in I)$
is a common zero of the polynomials $f_k\,$, $k\in I$, and $\min_{i\in
I} v(y_i-y'_i)\geq \alpha$.
\end{theorem}

%
%
\section{Polynomials in additive operators}
In this section, we will consider polynomials $f\in {\cal O}[X_0,X_1,
\ldots,X_n]$ over valued fields $(K,v)$ and additive operators
$\sigma_i: K\rightarrow K$, $0\leq i\leq n$. We write $\sigma=(\sigma_0,
\ldots,\sigma_n)$. We will try to solve equations in one variable of
the form
\[f^\sigma X\>:=\>f(\sigma_0 X,\sigma_1 X,\ldots,\sigma_n X)\>=\>0\;.\]

%
%
\subsection{A basic result}                \label{sectbr}
For any polynomial $f$ in $n+1$ variables over a field of arbitrary
characteristic, we denote by $f^{[\,\ul{i}\,]}$ its $\ul{i}$-th formal
derivative, where $\ul{i}=(i_0,\ldots,i_n)$ is a multi-index.
These polynomials are defined such that the following analogue of
(\ref{Tayl1}) holds in arbitrary characteristic:
\begin{equation}                            \label{Tayl1m}
f(b+\varepsilon)\>=\>f(b)+\sum_{i\in I}
f^{[\,\ul{i}\,]} (b) \varepsilon^{\ul{i}} \;\;\;\mbox{ for all }
b,\varepsilon\in K^{n+1}\;,
\end{equation}
where $I=\{0,1,\ldots,\deg f\}^{n+1}\setminus\{(0,\ldots,0)\}$ and
$\varepsilon^{\ul{i}}= \varepsilon_0^{i_0}\cdot\ldots\cdot
\varepsilon_n^{i_n}$. Note that if $\ul{i}=(0,\ldots,0,1,0,\ldots,0)$
with the $1$ in the $j$-th place, then
$f^{[\,\ul{i}\,]}=\frac{\partial f}{\partial X_j}(X_0,\ldots,X_n)$.

\begin{lemma}                               \label{lvf-psi}
Take $f\in {\cal O}[X_0,\ldots,X_n]$ and $b\in
{\cal O}^{n+1}$, $s\in {\cal O}$ such that
\[vs\>=\>\min_{0\leq i\leq n} v\frac{\partial f}{\partial X_i}(b)
\><\>\infty\;.\]
Then for all distinct $y=(y_0,\ldots,y_n)$ and
$z=(z_0,\ldots,z_n)$ in $\,b+ s{\cal M}^{n+1}$,
\begin{equation}                            \label{vf-psi}
v\left(f(y)-f(z)\>-\>\sum_{i=0}^n (y_i-z_i)
\frac{\partial f}{\partial X_i}(b)\right)\> >\> vs+
\min_{0\leq i\leq n} v(y_i-z_i)
\end{equation}
and
\begin{equation}                            \label{vf-f}
v(f(y)-f(z))\>\geq\> vs+\min_{0\leq i\leq n} v(y_i-z_i)\;.
\end{equation}
In particular,
\[f(b+ s{\cal M}^{n+1})\>\subseteq\>f(b)+ s^2{\cal M}^{n+1}\;.\]
\end{lemma}
\begin{proof}
Since $f\in {\cal O}[X_0,\ldots,X_n]$, we have that $f^{[\,\ul{i}\,]}
\in {\cal O} [X_0,\ldots,X_n]$. Since $b\in {\cal O}^{n+1}$, we
also have that $f^{[\,\ul{i}\,]}(b)\in {\cal O}$. Write $y=
b+\delta$ and $z=b+ \varepsilon$ with
$\delta=(\delta_0,\ldots,\delta_n), \varepsilon=
(\varepsilon_0,\ldots,\varepsilon_n)\in s{\cal M}^{n+1}$. Then by
(\ref{Tayl1m}),
\[f(y)-f(z)\>=\>\sum_{i=0}^n (\delta_i-
\varepsilon_i) \frac{\partial f}{\partial X_i}(b)\;+\;
\sum_{i\in I'} (\delta^{\ul{i}}-\varepsilon^{\ul{i}})
f^{[\,\ul{i}\,]} (b)\]
where $I'=\{\ul{i}\in I\mid |\ul{i}|\geq 2\}$ with
$|\ul{i}|:=i_0+\ldots+i_n\,$.

Choose $c\in {\cal M}$ such that $vc=\min_i v(\delta_i-\varepsilon_i)
=\min_i v(y_i-z_i)$. Pick $j\in\{0,\ldots,n\}$ and take $\ul{i}\in I'$
such that $i_j\ne 0$. Let $\ul{i}'\in I$ be the multi-index obtained
from $\ul{i}$ by subtracting $1$ in the $j$-th place.
Then
\[\delta^{\ul{i}}-\varepsilon^{\ul{i}}\>=\>
\delta_j \delta^{\ul{i}'} -\varepsilon_j\varepsilon^{\ul{i}'}
\>=\>(\delta_j-\varepsilon_j)\delta^{\ul{i}'}
+\varepsilon_j(\delta^{\ul{i}'}-\varepsilon^{\ul{i}'})\]
Suppose we have already shown by induction on $|\ul{i}|$ that
$\delta^{\ul{i}'}-\varepsilon^{\ul{i}'}\in c{\cal O}$. Since
$\delta_j-\varepsilon_j\in c{\cal O}$ and $\delta^{\ul{i}'},
\varepsilon_j\in s{\cal M}$, we then find that
\[\delta^{\ul{i}}-\varepsilon^{\ul{i}}\>\in\> s c{\cal M}\]
for every multi-index $\ul{i}$ with $|\ul{i}|\geq 2$. Since also
$f^{[\,\ul{i}\,]}
(b)\in {\cal O}$, we obtain that
\[f(y)-f(z)\>-\>\sum_{i=0}^n (\delta_i-\varepsilon_i)
\frac{\partial f}{\partial X_i}(b)\>=\>
\sum_{i\in I'} (\delta^{\ul{i}}-\varepsilon^{\ul{i}})
f^{[\,\ul{i}\,]} (b) \>\in\> s c{\cal M}\;.\]
This proves (\ref{vf-psi}). To prove (\ref{vf-f}), we observe that
\[v\, \sum_{i=0}^n (y_i-z_i)\frac{\partial f}{\partial X_i}(b)
\>\geq\>\min_{0\leq i\leq n} v(y_i-z_i)\frac{\partial f}{\partial X_i}
(b) \>\geq\>vs+ \min_{0\leq i\leq n} v(y_i-z_i)\]
and therefore,
\begin{eqnarray*}
\lefteqn{v(f(y)-f(z))\>\geq}\\
& \geq & \min\left\{v\left(f(y)-f(z)
\>-\>\sum_{i=0}^n (y_i-z_i) \frac{\partial f}{\partial X_i}(b)
\right)\,,\, v\, \sum_{i=0}^n (y_i-z_i)
\frac{\partial f}{\partial X_i}(b)\right\}\\
 & \geq & vs+\min_{0\leq i\leq n} v(y_i-z_i)\;.
\end{eqnarray*}
The last assertion is obtained by applying (\ref{vf-psi}) with $z=b$.
\end{proof}

\begin{proposition}                         \label{psi}
Take
\sn
$\bullet$ \ additive operators $\sigma_i: {\cal O}\rightarrow
{\cal O}\,$, $\>0\leq i\leq n$,
\sn
$\bullet$ \ $f\in {\cal O}[X_0,\ldots,X_n]$,
\sn
$\bullet$ \ $b\in {\cal O}$ such that at least one of the following
derivatives is not zero:
\begin{equation}                            \label{d_i}
d_i\>:=\>\frac{\partial f}{\partial X_i}(\sigma_0 b, \sigma_1
b,\ldots,\sigma_n b)\qquad (0\leq i\leq n),
\end{equation}
$\bullet$ \ $s\in {\cal O}$ such that
\begin{equation}                            \label{vs}
vs\>=\>\min_{0\leq i\leq n} vd_i\;.
\end{equation}
Suppose that
\sn
{\bf (V$\geq$)} \ $v\sigma_i a\>\geq\>va$ \ for all \ $a\in
{\cal O}\qquad (0\leq i\leq n)\,$
\sn
holds and that the additive operator
\[\phi\>:=\>\sum_{i=0}^{n}d_i\sigma_i:\;s{\cal M}\longrightarrow
s^2{\cal M}\]
has the property that for all $a'\in s^2{\cal M}$ there is some $a\in
s{\cal M}$ such that $v(a'-\phi a)>va'$ and $va=va'-vs$.
Then the maps $\phi$ and
\[b+s{\cal M}\ni x\;\mapsto\; f^\sigma x\in f^\sigma b +s^2{\cal M}\]
are immediate.
\end{proposition}
\begin{proof}
%
For all $a\in s{\cal M}$, the definition of $s$ together with
(V$\geq$) yields
\begin{equation}                            \label{vphia}
v\phi a\>=\>v\sum_{i=0}^{n}d_i\sigma_i a\>\geq\>
\min_{0\leq i\leq n} vd_i\sigma_i a\>\geq\>
\min_{0\leq i\leq n} vd_i \,+\,va\>=\>vs+va\;.
\end{equation}
%
%
We wish to apply Proposition~\ref{BCb} to the map $f^\sigma$.
Take distinct elements $y,z\in b+s{\cal M}$. From (V$\geq$) it follows
that $b_i:=\sigma_i b\in {\cal O}$, $y_i:=\sigma_i y\in
{\cal O}$, $z_i:= \sigma_i z\in {\cal O}$ with $y_i-b_i=\sigma_i(y-b)
\in s{\cal M}$ and $z_i-b_i=\sigma_i(z-b)\in s{\cal M}$, so
$(y_0,\ldots,y_n)\,,\, (z_0,\ldots,z_n)\in (b_0,\ldots,b_n)+
s{\cal M}^{n+1}$. Thus we can apply Lemma~\ref{lvf-psi} to
obtain
\begin{eqnarray*}
\lefteqn{v(f^\sigma y-f^\sigma z\,-\,\phi (y-z))\>=\>
v\left(f^\sigma y-f^\sigma z\,-\,\sum_{i=0}^n
d_i\sigma_i (y-z)\right)}\\
& = & v\left(f(\sigma_0 y,\ldots,\sigma_n y)
-f(\sigma_0 z,\ldots,\sigma_n z)-\sum_{i=0}^n
(\sigma_i y- \sigma_i z) \frac{\partial f}{\partial X_i}
(\sigma_0 b,\ldots,\sigma_n b)\right)\\
&>& vs+\min_i v(\sigma_i y-\sigma_i z) \>=\>vs+\min_i v\sigma_i
(y-z)\>\geq\> vs+v(y-z)\;.
\end{eqnarray*}
We also obtain that $f^\sigma(b+s{\cal M})\subseteq f^\sigma b
+s^2{\cal M}$.

Now take any $a'\in s^2{\cal M}$. By assumption, there is some $a\in
s{\cal M}$ such that $v(a'-\phi a)>va'$ and $va=va'-vs=v\phi a-vs$.
Take distinct elements $y,z\in b+s{\cal M}$ such that $v(y-z)\geq va$.
By what we have shown above, $v(f^\sigma y-f^\sigma z-\phi (y-z))>
vs+v(y-z)\geq vs+va =v\phi a$.

We have to show that $a\in\Reg(\phi)$. Indeed, if $va\leq vb$, then
$v\phi a = vs+va\leq vs+vb\leq v\phi b$ by (\ref{vphia}). It now follows
from Proposition~\ref{ATC} that $\phi$ is immediate, and from
Proposition~\ref{BCb} that $f^\sigma$ is immediate.
\end{proof}

In the next section, we give a criterion which guarantees that the
hypothesis of Proposition~\ref{psi} on the operator $\phi$ is satisfied.

%
%
\subsection{The case of operators compatible with a weak coefficient
map}                                        \label{sectwcm}
Let us start with the following useful observation.

\begin{lemma}                               \label{exco}
Let $(K,v)$ be any valued field. For all $\alpha\in vK$, choose elements
\begin{equation}                            \label{ma}
m_\alpha\in K \mbox{ \ such that \ } vm_\alpha\>=\>\alpha\mbox{ \
and \ } m_0=1\;.
\end{equation}
Define $\co 0:=0$ and
\[\co a \>:=\> (m_{-va}\,a)\,v\mbox{ \ for all \ } a\in K
\setminus \{0\}\;.\]
Then $\co$ has the following properties:
\sn
{\bf (WCM0)} \ $\co a=0$ if and only if $a=0$,
\n
{\bf (WCM1)} \ if $va=0$, then $\co a=av$,
\n
{\bf (WCM2)} \ if $va_1=va_2=\ldots=va_k$
and $\sum_{i=1}^{k}\co a_i\ne 0$, then $\co
(\sum_{i=1}^{k} a_i)=\sum_{i=1}^{k}\co a_i\,$,
\n
{\bf (WCM3)} \ if $\co a=\co b$ and $va=vb$, then $v(a-b)>va$,
\n
{\bf (WCM4)} \ if $\gamma\in vK$ and $0\ne\ovl{a}\in Kv$, then
$\exists a\in K: \co a=\ovl{a}$ and $va=\gamma$.
\end{lemma}
\begin{proof}
Since $(m_{-va}\,a)v\ne 0$ for $a\ne 0$, (WCM0) holds.
Since $m_0=1$, also (WCM1) holds.

\pars
If $va_1=va_2=\ldots=va_k$ and $\sum_{i=1}^{k}\co a_i\ne 0$, then
$m_{-va_1}=m_{-va_2}=\ldots= m_{-va_k}$ and
\[0\>\ne\> \sum_{i=1}^{k}\co a_i\>=\>\sum_{i=1}^{k}(m_{-va_i}\,a_i)\,v
\>=\>\sum_{i=1}^{k}(m_{-va_1}\,a_i)\,v\>=\>\left(m_{-va_1}\sum_{i=1}^{k}
a_i\right) v\;,\]
whence $vm_{-va_1}\sum_{i=1}^{k}a_i=0$ and therefore,
$v\sum_{i=1}^{k}a_i=va_1\,$. Hence,
\[\sum_{i=1}^{k} \co a_i  \>=\> \left(m_{-va_1}\sum_{i=1}^{k}
a_i\right) v \>=\> \co\left(\sum_{i=1}^{k}a_i\right)\;.\]
This shows that (WCM2) holds.

\pars
If $va=vb$ and $\co a=\co b$, then
\[(m_{-va}\,a)\,v\>=\>\co a\>=\>\co b\>=\>(m_{-vb}\,b)\,v
\>=\>(m_{-va}\,b)\,v\;,\]
so $0<v(m_{-va}a-m_{-va}b)=vm_{-va}+v(a-b)=-va+v(a-b)$, that is,
$v(a-b)>va$. This shows that (WCM3) holds.

\pars
If $\gamma\in vK$ and $0\ne\ovl{a}\in Kv$, we choose $a_0\in
{\cal O}^\times$ such that $a_0v=\ovl{a}$. Then we set
$a=m_{-\gamma}^{-1}a_0$. This gives $va=-vm_{-\gamma}=\gamma$ and
$\co a=(m_{-\gamma}(m_{-\gamma}^{-1}a_0))v =a_0v=\ovl{a}$.
Hence, (WCM4) holds.
\end{proof}
\n
A map $\co$ with properties (WCM0) -- (WCM4) will be called a
\bfind{weak coefficient map}. We will assume that the operators
$\sigma_i$ satisfy (V$\geq$); hence they induce additive operators
$\ovl{\sigma}_i$ on $Kv$:
\begin{equation}                            \label{s_iav}
\mbox{for all $a\in {\cal O}$,} \quad
\ovl{\sigma}_i(av)\>=\>(\sigma_i a)v \qquad (0\leq i\leq n)\;.
\end{equation}
We will need some stronger compatibility of the $\sigma_i$ with the weak
coefficient map:

\begin{lemma}                               \label{lIO}
Assume that the operators $\sigma_i$ satisfy {\rm (V$\geq$)} and that
the elements $m_\alpha$ in (\ref{ma}) can be chosen such that
\begin{equation}                            \label{IOc}
\mbox{for all $a\in {\cal O}$,}
\quad v(\sigma_i m_{-va} a\,-\,m_{-va}\sigma_i a)\>>\>0
\qquad (0\leq i\leq n)\;.
\end{equation}
Then
\begin{equation}                            \label{IO}
\mbox{for all $a\in {\cal O}$ and all $d\in {\cal O}^\times$,}
\quad (\co d)\, \ovl{\sigma}_i \,\co a \> = \left\{
\begin{array}{cl}
\co (d\sigma_i a) & \mbox{if \ } v\sigma_i a=va\\
0 & \mbox{if \ } v\sigma_i a>va
\end{array}\right.
\end{equation}
\end{lemma}
\begin{proof}
Take any $d\in {\cal O}^\times$; then $vd=0$ and hence, $\co d=dv$.
We have that
\begin{eqnarray*}
(\co d)\,\ovl{\sigma}_i \co a & = & (dv)\,\ovl{\sigma}_i
((m_{-va}a)v)\>=\>(dv)\,(\sigma_i m_{-va}a)v\\
 & = & (dv)\,(m_{-va}\sigma_i a)v\>=\>(m_{-va} d \sigma_i a)v\;.
\end{eqnarray*}
Here, the second equality holds by equation (\ref{s_iav}), and the third
equality holds by (\ref{IOc}). Now we distinguish two cases. Suppose
first that $v\sigma_i a=va$. Then
\[(m_{-va} d \sigma_i a)v\>=\> (m_{-v\sigma_i a} d \sigma_i a)v \>=\>
(m_{-vd\sigma_i a}d \sigma_i a)v\>=\> \co(d \sigma_i a)\;.\]
Now suppose that $v\sigma_ia>va$. Then $vm_{-va} d \sigma_i a>0$ and hence,
$(m_{-va} d \sigma_i a)v=0$. This proves that (\ref{IO}) holds.
\end{proof}
\n
Property (\ref{IO}) can be expressed by saying that unit multiples of
the additive operators commute with the coefficient map.

\begin{proposition}                         \label{immopwcm}
Let the assumptions on $f$, $b$, $d_i$ and $s$ be as in
Proposition~\ref{psi}. Assume that the additive operators $\sigma_i$
satisfy {\rm (V$\geq$)}, that $\co$ is a weak coefficient map and that
(\ref{IO}) holds.
Suppose further that the additive operator
\[\sum_{i=0}^{n} c_i \ovl{\sigma}_i \mbox{ \ \ with \ \ } c_i=
\left\{
\begin{array}{cl}
\co s^{-1}d_i & \mbox{if \ } vd_i=vs\\
0 & \mbox{if \ } vd_i>vs
\end{array}\right.\]
on the residue field $Kv$ is surjective.
Then the map
\[b+s{\cal M}\ni x\mapsto f^\sigma(x) \in f^\sigma(b)+s^2{\cal M}\]
is immediate.
\end{proposition}
\begin{proof}
We define $\phi$ as in Proposition~\ref{psi}. Now we just have to show
that $\phi$ satisfies the assumptions of that proposition. So take any
$a'\in s^2{\cal M}$, $a'\ne 0$. Since $\sum_{i=0}^{n}c_i\ovl{\sigma}_i$
is surjective on $Kv$ by assumption, there is some $\ovl{a}
\in Kv$ such that $\sum_{i=0}^{n} c_i \ovl{\sigma}_i\, \ovl{a}=\co
s^{-1}a'$. Property (WCM4) of the coefficient map allows us to choose
$a\in K$ such that $\co a=\ovl{a}$ and $va=va'-vs$. Thus, $0\ne
a\in s{\cal M}$. Set $I=\{i\mid 0\leq i\leq n\mbox{ with } vd_i=vs
\mbox{ and } \ovl{\sigma}_i\,\co a\ne 0\}$. Then by the definition of
the $c_i\,$,
\begin{eqnarray*}
\co s^{-1}a' & = & \sum_{i=1}^{n} c_i \ovl{\sigma}_i\,\ovl{a}\>=\>
\sum_{i\in I} \co (s^{-1}d_i)\,\ovl{\sigma}_i\,\co a\\
 & = &
\sum_{i\in I} \co (s^{-1}d_i \sigma_i a)\>=\>
\co (\sum_{i\in I} s^{-1}d_i \sigma_i a)\;,
\end{eqnarray*}
where the third equality holds by (\ref{IO}). The last equality follows
from (WCM2) since the left hand side is non-zero, being equal to $\co
s^{-1} a'$, and because for each $i\in I$, $\ovl{\sigma}_i\,\co a\ne 0$
implies $v\sigma_ia=va$ by (\ref{IO}), and $vd_i=vs$ then yields
$vs^{-1}d_i \sigma_i a=va$ so that all values are equal. By (WCM3), it
follows that
\[v\left(s^{-1}a'\,-\,\sum_{i\in I} s^{-1}d_i \sigma_i a\right)
\>>\> vs^{-1}a'\;.\]
Consequently,
\[v\left(a'\,-\,\sum_{i\in I} d_i \sigma_i a\right)
\>=\>v\left(s^{-1}a'\,-\,\sum_{i\in I} s^{-1}d_i \sigma_i a\right)+vs
\>>\> vs^{-1}a'+vs\>=\>va'\;.\]
On the other hand, take $i\in I':=\{0,\ldots,n\}\setminus I$. In the
case of $vd_i>vs$, since $v\sigma_i a\geq va=va'-vs$, we find that
$vd_i\sigma_i a \geq vd_i+va'-vs>va'$. Observe that $a\ne 0$ implies
$d\sigma_i a\ne 0$, and this implies $\co d\sigma_i a\ne 0$. Hence in
the case of $\ovl{\sigma}_i\,\co a=0$, (\ref{IO}) shows that $v\sigma_i
a>va$ and we obtain that $vd_i\sigma_i a > vd_i+va = vd_i+va'-vs\geq
va'$. Therefore,
\[v\,\sum_{i\in I'} d_i \sigma_i a \>\geq\>
\min_{i\in I'} vd_i \sigma_i a \>>\> va'\;.\]
This gives us
\[v(a'-\phi a) \>=\> v\left(a'\,-\,\sum_{i=0}^n d_i \sigma_i a\right)
\>\geq\> \min\left\{v\left(a'\,-\,\sum_{i\in I} d_i \sigma_i a\right),
v\,\sum_{i\in I'} d_i \sigma_i a\right\}\>>\>va'\;.\]
So the conditions of Proposition~\ref{psi} are satisfied and we are
done.
\end{proof}

In the same way as for the original Hensel's Lemma (except for the
uniqueness assertion), Proposition~\ref{immopwcm} yields the following
generalized Hensel's Lemma in the present setting:
\begin{theorem}                             \label{genHLsao}
In addition to the assumptions of Proposition~\ref{immopwcm}, suppose
that $(K,v)$ is spherically complete and that
\[vf^\sigma(b)>2vs\;.\]
Then there is an element $a\in K$ such that $f^\sigma(a)=0$ and
$v(a-b)>vs$.
\end{theorem}

%
%
\subsection{The case of a dominant operator}        \label{sectdom}
In this section, we consider the case where one of the additive
operators, say $\sigma_n$ (without loss of generality), is dominant on
some ball $B$ around $0$, that is,
\begin{equation}                            \label{domin}
\forall \,a\in B:\;\;\; v\sigma_n a \,< \min_{0\leq j\leq n-1}
v\sigma_j a \mbox{ \ \ or \ \ } \sigma_0 a=\sigma_1 a=\ldots=\sigma_n
a=0 \;.
\end{equation}
We will not assume that (V$\geq$) holds, so we cannot apply
Proposition~\ref{psi}. Instead, we prove:

\begin{proposition}                         \label{immopdom}
Let $\sigma_i:{\cal O}\rightarrow {\cal O}$, $0\leq i\leq n$, be
additive operators satisfying condition (\ref{domin}). With $f$, $b$ and
$d_i$ as in Proposition~\ref{psi}, assume that
\begin{equation}                            \label{dnmindi}
vd_n\>=\>\min_{0\leq i\leq n} vd_i\;.
\end{equation}
Suppose further that for some balls $B,B'\subseteq d_n {\cal M}$ around
$0$, the map $\sigma_n: B\rightarrow B'$ is immediate. Then the map
\begin{equation}                            \label{embdom}
b+B\ni x\mapsto f^\sigma x \in f^\sigma b+d_n B'
\end{equation}
is immediate. If $\sigma_n$ is injective on $B$, then
(\ref{embdom}) is injective, too.
\end{proposition}
\begin{proof}
We set $s=d_n\,$. Take distinct elements $y,z\in b+B\subseteq b+s{\cal M}$
and set $b_i:= \sigma_i b\in {\cal O}$, $y_i:=\sigma_i y\in {\cal O}$,
$z_i:=\sigma_i z\in {\cal O}$. It follows from (\ref{domin}) that
$v(y_i-b_i)= v\sigma_i (y-b)\geq v\sigma_n(y-b)$, and our assumption on
$\sigma_n$ yields that $y_i-b_i\in B'$ for $0\leq i\leq n$. We obtain
$y_i\in b_i+B'\subseteq b_i +s{\cal M}$ and similarly, $z_i\in b_i
+s{\cal M}$. Thus we can apply Lemma~\ref{lvf-psi} to obtain
that
\[f^\sigma (b+B) \>\subseteq\> f^\sigma b +d_n B'\;.\]
We shall apply Proposition~\ref{BCb} in order to show that $g: b+B
\rightarrow f^\sigma b +d_n B'$ is immediate. We set
$\phi:= d_n\sigma_n\,$. Pick any $a'\in d_n B'$, $a'\ne 0$. Since
$\sigma_n: B\rightarrow B'$ is immediate, Proposition~\ref{ATC} shows
that there is some $a\in B$ such that $a\ne 0$ and
\begin{equation}                            \label{a1}
v\left(\frac{a'}{d_n}-\sigma_n a\right)\> >\>v\,\frac{a'}{d_n}
\end{equation}
and
\begin{equation}                            \label{a2}
va\leq vb\;\Longrightarrow\; v\sigma_n a\leq v\sigma_n b\;.
\end{equation}
We obtain that $v(a'-\phi a)>va'$ and $va\leq vb\rightarrow v\phi a\leq
v\phi b$, which shows that (\ref{(BC1)}) of Proposition~\ref{BCb} is
satisfied. Now take distinct $y,z\in b+B$. As in the proof of
Proposition~\ref{psi}, we can apply Proposition~\ref{lvf-psi}
to obtain that
\[v\left(f^\sigma y-f^\sigma z\,-\,\sum_{i=0}^n d_i\sigma_i (y-z)\right)
\>>\>vs+\min_i v(\sigma_i y-\sigma_i z)\>=\>vs+\min_i v\sigma_i(y-z)\;.\]
By (\ref{domin}),
\[vs+\min_i v\sigma_i (y-z) \>=\>vs+v\sigma_n (y-z)\>=\>
vd_n\sigma_n (y-z)\;.\]
Again by (\ref{domin}),
\[v\sum_{i=0}^{n-1} d_i\sigma_i (y-z)\>>\>vd_n\sigma_n (y-z)\;,\]
and we conclude that
\begin{equation}                            \label{gy-gz}
v(f^\sigma y-f^\sigma z\,-\,d_n\sigma_n(y-z))\>\geq\>
\min\{v(f^\sigma y-f^\sigma z-\sum_{i=0}^n
d_i\sigma_i (y-z))\,,\,\sum_{i=0}^{n-1} d_i\sigma_i (y-z)\}
\>>\>vd_n\sigma_n (y-z)\;.
\end{equation}
If $v(y-z)\geq va$, then by (\ref{a2}), $vd_n\sigma_n (y-z)\>\geq\>
vd_n\sigma_n a$ and thus, (\ref{gy-gz}) yields
\[v(f^\sigma y-f^\sigma z\,-\,\phi(y-z))\>=\>v(f^\sigma y-f^\sigma z
\,-\,d_n\sigma_n(y-z))\>>\>
vd_n\sigma_n a\>=\>v\phi a\;.\]
Since $\phi 0=0$ as $\phi$ is additive, this shows that (\ref{(BC2)}) is
satisfied for $f^\sigma $ in the place of $f$. Now Proposition~\ref{BCb}
proves that $f^\sigma $ is immediate.

If $\sigma_n$ is injective on $B$, then $y\ne z$ implies $vd_n\sigma_n
(y-z) <\infty$, whence $f^\sigma y\ne f^\sigma z$ by (\ref{gy-gz}).
Hence in this case, (\ref{embdom}) is injective.
\end{proof}

Proposition~\ref{immopdom} yields the following Hensel's
Lemma for the case of a dominant operator:

\begin{theorem}                             \label{genHLdom}
In addition to the assumptions of Proposition~\ref{immopdom}, suppose
that $(K,v)$ is spherically complete and that for some $e\in B$,
\begin{equation}                            \label{genHLdomc}
vf^\sigma b\geq vd_n+v\sigma_n e\;.
\end{equation}
Then there is an element $a\in b+B$ such that $f^\sigma a=0$ and
$v\sigma_n(a-b)\geq v\sigma_n e$. If $\sigma_n$ is injective on $B$,
then $a$ is unique.
\end{theorem}
\begin{proof}
It just remains to show that $v\sigma_n(a-b)\geq v\sigma_n e$. By
(\ref{genHLdomc}),
\[vd_n+v\sigma_n e \>\leq\> vf^\sigma b\>=\>v(f^\sigma b\,-\,f^\sigma a)
\>=\> vd_n+v\sigma_n(b-a)\;,\]
where the last equality follows from (\ref{gy-gz}) by the ultrametric
triangle law. Hence, $v\sigma_n(a-b)=v\sigma_n(b-a)\geq v\sigma_n e$.
\end{proof}

In Section~\ref{sectde} we will deduce from this theorem a Hensel's
Lemma for Rosenlicht valued differential fields. But this Hensel's Lemma
is not strong enough. To improve it, we consider also the values of the
higher derivatives of~$f$. So we need to modify our approach, which we
will do in the next section.

%
%
\subsection{Rosenlicht systems of operators}     \label{sectRso}
We will call $\sigma_0,\sigma_1,\ldots,\sigma_n$ a \bfind{Rosenlicht
system of operators} if each $\sigma_i:{\cal O}\rightarrow {\cal O}$ is
additive and there exist elements $e_i\in {\cal O}$ such that
\begin{equation}                            \label{vei}
e_n\>=\>1 \mbox{ \ \ and \ \ } ve_0\>\geq\>ve_1\>\geq\ldots\geq\>
ve_n\>=\>0\;,
\end{equation}
and for all $i<n$,
\begin{equation}                            \label{Rsys}
ve_i+v\sigma_i a\>>\>v\sigma_n a\;\;\;\mbox{ for all } a\in {\cal M},
\> a\ne 0\;.
\end{equation}
The latter implicitly includes the condition that $\sigma_n$ is
injective on ${\cal M}$.

\parm
The following is an adaptation of Lemma~\ref{lvf-psi}.
\begin{lemma}                               \label{Rlvf-psi}
Take $f\in {\cal O}[X_0,X_1,\ldots,X_n]$ and $b\in
{\cal O}^{n+1}$ such that
\[d_n\>=\>\frac{\partial f}{\partial X_n}(b)\>\ne\>0\]
and for all $\>\ul{i}\in I=\{0,\ldots,\deg f\}^{n+1}\setminus
\{(0,\ldots,0)\}$,
\begin{equation}                            \label{vfi}
vf^{[\,\ul{i}\,]} (b)\>\geq\>vd_n+ve_k\;\;\;\mbox{ if \ }
k= \min\{j\mid i_j \ne 0\}
\end{equation}
where the elements $e_i\in K$ satisfy (\ref{vei}).
Take $y=(y_0,\ldots,y_n)$ and $z=(z_0,\ldots,z_n)$ in
$b+{\cal M}^{n+1}$ such that
\begin{equation}
ve_i+v(y_i-z_i)\>>\>v(y_n-z_n)\;\;\;\mbox{ for \ } 0\leq i<n\;.
\end{equation}
Then the following holds:
\begin{equation}                            \label{Rvf-psi}
v(f(y)-f(z)\>-\>d_n(y_n-z_n))\> >\> vd_n(y_n-z_n)\>=\>
v(f(y)-f(z))\;.
\end{equation}
\end{lemma}
\begin{proof}
Write $y=b+\delta\in b+{\cal M}^{n+1}$ and $z=
b+ \varepsilon\in b+{\cal M}^{n+1}$, where $\delta
=(\delta_0,\ldots,\delta_n)$ and $\varepsilon=(\varepsilon_0,
\ldots,\varepsilon_n)$ satisfy
\begin{equation}                            \label{vede>}
ve_i+v(\delta_i-\varepsilon_i)\>=\>ve_i+v(y_i-z_i)\>>\>v(y_n-z_n)
\;\;\;\mbox{ for \ } 0\leq i<n\;.
\end{equation}
We note that $ve_n+v(\delta_n-\varepsilon_n)=v(y_n-z_n)$; so we have
\begin{equation}                            \label{vede}
ve_i+v(\delta_i-\varepsilon_i)\>\geq\>v(y_n-z_n)
\;\;\;\mbox{ for \ } 0\leq i\leq n\;.
\end{equation}
Take $\ul{i}\in I$, $|\ul{i}|\geq 2$, and let $\ul{i}'$
be the multi-index obtained from $\ul{i}$ by subtracting $1$ in the
$k$-th place, where $k= \min\{j\mid i_j \ne 0\}$.
Then
\[\delta^{\ul{i}}\,-\,\varepsilon^{\ul{i}}
\>=\>(\delta_k-\varepsilon_k)\delta^{\ul{i}'}
+\varepsilon_k(\delta^{\ul{i}'}-\varepsilon^{\ul{i}'})\;.\]
%
Suppose that we have already shown by induction on $|\ul{i}'|$ that
\[ve_{\ell}+v(\delta^{\ul{i}'}-\varepsilon^{\ul{i}'})\>\geq
\> v(y_n-z_n)\;\;\; \mbox{ for }\ell=\min\{j\mid i'_j \ne 0\}\;,\]
with the induction start for $|\ul{i}'|=1$ being covered by
(\ref{vede}). We have that $\ell\geq k$, hence $ve_k\geq ve_{\ell}$ by
(\ref{vei}); therefore, also $ve_k+v(\delta^{\ul{i}'}-
\varepsilon^{\ul{i}'}) \geq v(y_n-z_n)$. Since
$ve_k+v(\delta_k-\varepsilon_k)\geq v(y_n-z_n)$ by (\ref{vede}),
and since $\delta^{\ul{i}'}\,,\,\varepsilon_k\in {\cal M}$,
we then find
\begin{eqnarray}
ve_k+v(\delta^{\ul{i}}-\varepsilon^{\ul{i}}) &\geq&
\min\{ve_k+v(\delta_k-\varepsilon_k)+v\delta^{\ul{i}'}\,,\,
ve_k+v\varepsilon_k+v(\delta^{\ul{i}'}-\varepsilon^{\ul{i}'})
\}\nonumber\\
&>& v(y_n-z_n)\;.                       \label{yn-zn}
\end{eqnarray}
%
Take $\ul{i}\in I':= I\setminus \{(0,\ldots,0,1)\}$. Then because of
(\ref{vede>}), inequality (\ref{yn-zn}) also holds in the case of
$|\ul{i}|=1$. Hence by hypothesis (\ref{vfi}),
\[v(\delta^{\ul{i}}-\varepsilon^{\ul{i}})f^{[\,\ul{i}\,]}
(b)\>\geq\>vd_n+ve_k+v(\delta^{\ul{i}}\,-\,
\varepsilon^{\ul{i}}) \>>\>vd_n+v(y_n-z_n)\;.\]
Since
\[f(y)-f(z)\>=\>d_n (\delta_n-\varepsilon_n) \;+\;
\sum_{i\in I'} (\delta^{\ul{i}}\,-\,\varepsilon^{\ul{i}})
f^{[\,\ul{i}\,]} (b)\]
by (\ref{Tayl1m}), this yields
\begin{eqnarray*}
v(f(y)-f(z)\>-\>d_n(y_n-z_n))
&=& v(f(y)-f(z)\>-\>d_n(\delta_n-\varepsilon_n))\\
&=& v\sum_{i\in I'}(\delta^{\ul{i}}\,-\,\varepsilon^{\ul{i}})
f^{[\,\ul{i}\,]} (b) \>>\> vd_n+v(y_n-z_n)\;,
\end{eqnarray*}
which gives the inequality in (\ref{Rvf-psi}). The equality in
(\ref{Rvf-psi}) follows from the inequality by the ultrametric triangle
law.
\end{proof}

\begin{proposition}                         \label{immopRos}
Let $\sigma_0,\ldots,\sigma_n$ be a Rosenlicht system of operators
satisfying (\ref{vei}) and (\ref{Rsys}). Take $f$, $b$ and $d_n$ as in
Proposition~\ref{Rlvf-psi} such that (\ref{vfi}) holds.
Suppose further that for some balls $B,B'\subseteq {\cal M}$ around $0$,
the map $\sigma_n: B\rightarrow B'$ is immediate. Then
\begin{equation}                            \label{embRos}
b+B\ni x\mapsto f^\sigma x \in f^\sigma b +d_n B'
\end{equation}
is immediate and injective.
\end{proposition}
\begin{proof}
We modify the proof of Proposition~\ref{immopdom} as follows.
In order to apply Lemma~\ref{Rlvf-psi}, we set $y_i=\sigma_i y$
and $z_i=\sigma_i z$. From (\ref{Rsys}) it follows that
\begin{eqnarray*}
ve_i+v(y_i-z_i) & = & ve_i+v(\sigma_i y -\sigma_i z)\>=\>ve_i+v\sigma_i
(y-z)\\
 & > & v\sigma_n (y-z)\>=\>v(\sigma_n y -\sigma_n z)\>=\>v(y_n-z_n)
\end{eqnarray*}
for all $y,z\in b+B$ and $0\leq i< n$. Therefore, we can apply
Lemma~\ref{Rlvf-psi}, and (\ref{Rvf-psi}) shows that
\begin{equation}                            \label{vmsig}
v(f(y)-f(z))\>=\> vd_n\sigma_n(y-z)\>=\>
vd_n+v\sigma_n(y-z)
\end{equation}
for all $y,z\in b+B$. It follows that
\[f^\sigma (b+B)\>\subseteq\> f^\sigma b+d_n B'\;.\]
As in the proof of Proposition~\ref{immopdom} we use
Proposition~\ref{BCb} to show that $f^\sigma: b+B\rightarrow f^\sigma
b+d_n B'$ is immediate. The proof that (\ref{(BC1)}) and (\ref{(BC2)})
hold can be taken over literally, except that instead of deducing
(\ref{gy-gz}) we just apply inequality (\ref{Rvf-psi}) of
Lemma~\ref{Rlvf-psi} to obtain that
\[v(f^\sigma y-f^\sigma z\,-\,d_n\sigma_n (y-z))\>>\>
vd_n\sigma_n (y-z)\;.\]

Since $\sigma_n$ is injective on ${\cal M}$ (as a consequence of
condition (\ref{Rsys})), it follows as in the proof of
Proposition~\ref{immopdom} that $g$ is injective.
%
\end{proof}

Proposition~\ref{immopRos} yields the following generalized Hensel's
Lemma for the case of a Rosenlicht system of operators:

\begin{theorem}                             \label{genHLRos}
The assertion of Theorem~\ref{genHLdom} also holds under the
assumptions of Proposition~\ref{immopRos}.
\end{theorem}

%
%
\section{Immediate differentiation}
From now on our $\sigma_i$ will be the $i$-th iterates $D^i$ of an
additive operator $D$, with $D^0$ being the identity. For a polynomials
$f$ in $n+1$ variables, we set $f^D(X)=f(X,DX,D^2X,\ldots,D^nX)$.

%
%
\subsection{VD-fields}                       \label{sectdf}
We will call a valued field $(K,v)$ with an additive map $D: K
\rightarrow K$ a \bfind{VD-field} if the following
conditions are satisfied:
\sn
{\bf (VDF1)} \ $vDa\geq va$ for all $a\in K$,
\n
{\bf (VDF2)} \ $vK=\{va\mid a\in K\mbox{ with } vDa> va\}$,
\n
{\bf (VDF3)} \ there is $e\in {\cal O}$ such that
$D(ab)=aDb+bDa+e(Da)(Db)$ for all $a,b\in K$.
\sn
Together with (VDF1), the additivity of $D$ implies:
\n
{\bf (VDF4)} \ $D$ induces an additive map on $Kv$, again
denoted by $D$, such that $(Da)v=D(av)$,

\pars
\begin{proposition}                         \label{Dimmsurj}
Let $(K,D,v)$ be a VD-field. Then $D$ is immediate if
and only if $D$ is surjective on $Kv$.
\end{proposition}
\begin{proof}
``$\Rightarrow$'': \ Take any $a'\in {\cal O}$; we have to show that
$D(av)=a'v$ for some $a\in {\cal O}$. Condition (IH1) implies that
there is $a\in K$ such that $v(a'-Da)>va'\geq 0$, whence
$a'v=(Da)v=D(av)$.

\pars
``$\Leftarrow$'': \ Take any $a'\in K\setminus\{0\}$. By (VDF2), we
choose $c\in K$ such that $vc=va'$ with $vDc>vc$, and set $a'_0=a'/c$.
Then $va'_0=0$, and since $D$ is surjective on $Kv$, there is some
$a_0\in {\cal O}$ such that $a'_0v=D(a_0v)=(Da_0)v$. Hence,
$v(a'_0-Da_0)>0$. We set $a=ca_0\,$. We have that $va_0Dc=va_0+vDc\geq
vDc>vc$ and $ve(Dc)(Da_0)=ve+vDc+vDa_0 \geq vDc>vc$. Hence,
\begin{eqnarray*}
v(a'-Da) & = & v(ca'_0-Dca_0)\>=\>v(ca'_0-cDa_0-a_0Dc-e(Dc)(Da_0))\\
 & \geq & \min\{vc+v(a'_0-Da_0)\,,\,va_0Dc\,,\,ve(Dc)(Da_0)\}
\> >\>vc\>=\>va'\;.
\end{eqnarray*}
This shows that (IH1) holds. Since $D(a_0v)=a'_0v\ne 0$, we know that
$a_0v\ne 0$, that is, $va_0=0$. Therefore, $vDa=va'=vc=vca_0=va$.
So we obtain from (VDF1) that $va\leq vb$ implies $vDa=va\leq vb\leq
vDb$, for all $b\in K$. Hence, also (IH2) is satisfied.
\end{proof}

The next theorem is an immediate consequence of this proposition
and Theorem~\ref{MT}.

\begin{theorem}                             \label{DF1}
Let $(K,D,v)$ be a spherically complete VD-field. Assume that $D$
is surjective on $Kv$. Then $D$ is surjective on $K$.
\end{theorem}

\parm
As a preparation for our ``$D$-Hensel's Lemma'', we need the
following facts:

\begin{lemma}                               \label{D1=0}
In every VD-field, $D1=0$.
\end{lemma}
\begin{proof}
Suppose that $D1\ne 0$. From (VDF3) with $b=1$ we then obtain $eDa=-a$
for all $a\in K$. With $a=1$ this yields $e=-(D1)^{-1}$, so $ve\leq 0$
since $vD1\geq v1=0$ by (VDF1). But by (VDF2), $e\in {\cal O}$, so we
get $ve=0$. But then $eDa=-a$ shows that $vDa=va$ for all $a\in K$,
in contradiction to (VDF2).
\end{proof}

Recall that by $D^i$ we denote the $i$-th iterate of $D$, with $D^0$
being the identity map.

\begin{lemma}                               \label{lvDm}
Let $(K,v)$ be a VD-field and $m\in K$ such that $vDm>vm$. Then
\begin{equation}                            \label{vDm>vm}
v\left(D^i(ma)-mD^ia\right)\>>\>vma
\end{equation}
for all $a\in K^\times$, and
\begin{equation}                            \label{vD1/m}
vDm^{-1}\>>\>vm^{-1}\;.
\end{equation}
\end{lemma}
\begin{proof}
By assumption, $vaDm=va+vDm>va+vm=vma$ and $ve(Dm)(Da)=ve+vDm+vDa\geq
vDm+va>vm+va=vma$. Hence by (VDF3),
\[v\left(D(ma)-mDa\right)\geq\min\{vaDm\,,\,ve(Dm)(Da)\}\>>\>vma\;.\]
Now we proceed by induction on $i$. Suppose that $j>1$ and that we have
already shown (\ref{vDm>vm}) for all $i<j$ and all $a\in K$. Then
\begin{eqnarray*}
\lefteqn{v\left(D^j(ma)-mD^ja\right)\>=}\\
& = & v\left(DD^{j-1}(ma)-D(mD^{j-1}a)+
D(mD^{j-1}a)-mDD^{j-1}a\right)\\
& \geq & \min\{vD\left(D^{j-1}(ma)
-mD^{j-1}a\right),v\left(D(mD^{j-1}a)-mDD^{j-1}a\right)\}\\
 & > & \min\{vma\,,\,vmD^{j-1}a\}\>=\>vma
\end{eqnarray*}
since $vD^{j-1}a\geq va$. This proves (\ref{vDm>vm}).

\parm
By Lemma~\ref{D1=0} and (VDF3),
\[0\>=\>D1\>=\>D(mm^{-1})\>=\>mDm^{-1}+m^{-1}Dm+e(Dm)(Dm^{-1})\;.\]
From this together with $veDm\geq vDm>vm$, we infer
\[vDm^{-1}\>=\>vm^{-1}Dm-v(m+eDm)\>=\>vm^{-1}+vDm-vm\>>\>vm^{-1}\;,\]
which proves (\ref{vD1/m}).
\end{proof}

In every VD-field, condition (V$\geq$) holds for the additive
operators $\sigma_i=D^i$. This follows by induction on $i$ (and we
have used it already in the last proof). Again by induction on $i$,
(VDF4) implies that
\begin{equation}                            \label{D^iav}
(D^ia)v\>=\>D^i(av)\;\;\;\mbox{ for every }i\geq 1\;,
\end{equation}
that is, the map induced by $D^i$ on $Kv$ is the $i$-th iterate of the
map induced by $D$ on $Kv$. Indeed, having already shown that
$(D^{i-1}a)v=D^{i-1}(av)$, we obtain
$(D^ia)v=(D(D^ia))v=D((D^ia)v)=D(D^i(av))=D^i(av)$.


Now we can prove the following theorem:
\begin{theorem}                             \label{DHLp}
Let $(K,D,v)$ be a spherically complete VD-field. Take a
polynomial $f\in {\cal O}[X_0,X_1,\ldots,X_n]$ and assume that
\sn
1) there is $b\in {\cal O}$ and $s\in K$ with $vDs>vs$ such that
\[vs=\min_{0\leq i\leq n} v\frac{\partial f}{\partial X_i}(b,
Db,\ldots, D^nb)<\infty \mbox{ \ \ and \ \ } vf^D b >2vs\;,\]
2) the additive operator
\begin{equation}                            \label{opciDi}
\sum_{i=0}^{n} c_i D^i \mbox{ \ \ with \ \ } c_i=
\left(s^{-1}\frac{\partial f}{\partial X_i}(b,Db\ldots,D^nb)\right) v
\end{equation}
on the residue field $Kv$ is surjective.
\sn
Then there is an element $a\in K$ such that $f^D a=0$ and $v(a-b)>vs$.
\end{theorem}
\begin{proof}
By (VDF2), we can choose elements $m_\alpha$ with $vm_\alpha=\alpha$ and
$vDm_\alpha>vm_\alpha$ for $\alpha\in vK$; we set $m_0=1$. By
Lemma~\ref{exco}, this gives rise to a weak coefficient map $\co$.
Inequality (\ref{vDm>vm}) of Lemma~\ref{lvDm} shows that condition
(\ref{IOc}) of Lemma~\ref{lIO} holds for the elements $m_\alpha$ and the
additive operators $\sigma_i= D^i$. Therefore, $\co$ satisfies
(\ref{IO}) for these operators. Since $vDs>vs$, inequality (\ref{vD1/m})
of Lemma~\ref{lvDm} shows that $vDs^{-1}>vs^{-1}$. Thus, we can choose
$m_{vs}=s$ and obtain that $\co a=(s^{-1}a)v$ whenever $va=vs$. With
$d_i$ defined as in Proposition~\ref{psi}, we thus obtain that the
elements $c_i$ defined above coincide with the elements $c_i$ defined in
Proposition~\ref{immopwcm} and that the operator $\sum_{i=0}^{n} c_i
D^i$
coincides with the operator $\sum_{i=0}^{n} c_i \ovl{\sigma}_i$ of
Proposition~\ref{immopwcm}. The former being surjective on $Kv$, our
theorem now follows from Theorem~\ref{genHLsao}.
\end{proof}

This theorem yields Theorem~\ref{DHL}. Indeed, if the assumptions of
that theorem are satisfied, then by use of (VDF2) we pick $s\in K$ with
$vDs>vs$ such that $vs=\gamma$. Since $Kv$ is assumed to be linearly
$D$-closed, the operator (\ref{opciDi}) on $Kv$ is surjective, and we
can apply Theorem~\ref{DHLp}.

%
%
\subsection{Integration on Rosenlicht valued differential
fields}                                       \label{sectrvdf}
Let $(K,D)$ be a differential field with field of constants
$C=\{a\in K\mid Da=0\}$. Following M.~Rosenlicht [R1], a valuation $v$
of $K$ is called a \bfind{differential valuation} if $C$ is a field of
representatives for the residue field of $(K,v)$ (that is, $v$ is
trivial on $C$ and for every $y\in K$ with $vy=0$ there is a unique
$c\in C$ s.t.\ $v(y-c)>0$), and $v$ satisfies
\begin{equation}                            \label{diffvdef}
\forall a,b\in K:\; va\geq 0\,\wedge\,vb>0\,\wedge\,b\ne 0\;
\Rightarrow\;v\left(\frac{bDa}{Db}\right)\,>\,0\;.
\end{equation}
Because of our assumption on $C$, this condition is equivalent to
\begin{equation}                            \label{diffv}
\forall a,b\in K\setminus\{0\}, va\ne 0, vb\ne 0:\;
va\leq vb\>\Leftrightarrow\>vDa\leq vDb\;.
\end{equation}


\begin{lemma}                               \label{fofc}
Assume that $v$ is a differential valuation with respect to $D$. Then
for every $\tilde{a}\in K$ there is some $a\in K$ such that $va\ne 0$
and $Da=D\tilde{a}$. Moreover,
\begin{equation}                            \label{diffv1}
\forall a,b\in K:\;
(0\ne va\wedge va\leq vb)\>\Rightarrow\>vDa\leq vDb\;.
\end{equation}
This shows that $\{a\in K\mid va\ne 0\}\subseteq\Reg(D)$.
\end{lemma}
\begin{proof}
If $v\tilde{a}=0$ then by our assumption that the field of constants is
a field of representatives for the residue field, there is some constant
$c$ such that $v(\tilde{a}-c)>0$; hence for $a:= \tilde{a}-c$ we have
that $va\ne 0$ and $Da=D\tilde{a}-Dc=D\tilde{a}$.

To prove (\ref{diffv1}), assume that $0\ne va$ and $va\leq vb$. If
$vb=0$, then we choose a constant $c$ such that $v(b-c)> 0$. So we can
infer from (\ref{diffv}) that $vDa\leq vD(b-c)=vDb$.
\end{proof}

\begin{proposition}                         \label{dv=tc}
Let $v$ be a differential valuation on $(K,D)$. Then $D:\>
(K,v) \rightarrow (K,v)$ is immediate if and only if
$(K,D,v)$ admits asymptotic integration.
\end{proposition}
\begin{proof}
``$\Rightarrow$'': \ Condition (IH1) implies that
$(K,D,v)$ admits asymptotic integration.
\sn
``$\Leftarrow$'': \ Take any $a'\in K\setminus\{0\}$. Since $(K,D,v)$
admits asymptotic integration, there is some $a\in K$ such that
$v(a'-Da)>va'$, that is, (IH1) holds. By Lemma~\ref{fofc}, $a$
can be chosen such that $va\ne 0$ and (IH2) holds.
\end{proof}

The next theorem is an immediate consequence of this proposition
and Theorem~\ref{MT}.

\begin{theorem}                             \label{D1}
Let $(K,D)$ be a differential field, endowed with a
spherically complete differential valuation $v$. Assume further that
$(K,D)$ admits asymptotic integration. Then $(K,D)$ admits integration.
\end{theorem}

\pars
For certain applications, one has to work with a field $K$ which is a
union of an increasing sequence of power series fields $K_i\,$,
$i\in\N$. If this sequence does not become stationary, then $K$ itself
will not be spherically complete. However, we still can prove the
following:
\begin{theorem}                             \label{LE}
Let $(K,v)$ be the union of an increasing chain $(K_i,v)$ of spherically
complete valued fields, $i\in\N$. Let $D$ be a derivation on $K$ such
that $v$ is a differential valuation with respect to $D$. Assume further
that for each $i$ there are elements $a_{i,j}\in K_{i+1}\,$, $j\in
I_i\,$, such that
\sn
1) \ $Da_{i,j}\in K_i\,$ for all $j\in I_i\,$,
\n
2) \ the valued $K_i$-subvector space $V_i:=K_i+\sum_{j\in I_i}
K_ia_{i,j}$ of $K_{i+1}$ is spherically complete,
\n
3) \ for every $b\in K_i$ there is some $a\in V_i$ such that
$v(b-Da)>vb$.
\sn
Then $(K,D)$ admits integration.
\end{theorem}
\begin{proof}
It suffices to show that for each $i$, $D$ is a surjective map from
$V_i$ onto $K_i\,$. Since $K=\bigcup_{i\in\N}^{}K_i$ it then follows
that $D$ is surjective on $K$.

Because of 1), we have that $DV_i\subseteq K_i\,$. We set $Y=V_i$ and
$Y'=K_i\,$. As in the proof of Proposition~\ref{dv=tc} one uses 3) to
show that $D:Y \rightarrow Y'$ is immediate. From 2) together with
Theorem~\ref{MT}, one obtains that $DV_i=K_i\,$.
\end{proof}

This theorem implies that the derivation on the
logarithmic-exponential power series field $\R((t))^{LE}$ (cf.\ [DMM3])
is surjective. The argument is as follows. It can be shown that
$\R((t))^{LE}$ is the union over an increasing sequence of differential
power series fields $K_i\,$ such that for every $i$ there is just one
$a_i\in K_{i+1}$ such that $Da_i\in K_i$ and condition 3) holds. In
fact, $a_i=\log_i x$ for a certain element $x$, where $\log_i$ denotes
the $i$-th iterate of
$\log$. Further, $va_i$ is rationally independent over $vK$. It follows
that $v(c+c'a_i)=\min\{vc, vc'a_i\}$ for all $c,c'\in K_i\,$, that is,
the ultrametric space underlying $V_i$ is just the direct product of the
one underlying $K_i$ and the one underlying $K_i a_i\,$. As the latter
is isomorphic to the one underlying $K_i\,$, both are spherically
complete. By Proposition~\ref{prodsphc}, their direct product is
spherically complete. The foregoing theorem now proves the surjectivity
of $D$.

%
%
%
\subsection{Differential equations on Rosenlicht valued
differential fields}                            \label{sectde}
Now let us assume in addition that
\begin{equation}                            \label{DM}
D({\cal M})\>\subseteq\> {\cal M}\;.
\end{equation}
If $K$ contains an element $x$ such that $vDx=0$ and $vx<0$ (as it is
the case in $\R((t))^{LE}$, see below), then (\ref{DM}) is a consequence
of (\ref{diffv}). In fact, (\ref{DM}) also holds in every Hardy field.
If (\ref{DM}) does not hold for a derivation $D$, then we may replace
$D$ by the derivation $aD$, with $0\ne a\in K$; it follows from
(\ref{diffv}) that (\ref{DM}) will hold for $aD$ in the place
of $D$ for every $a$ of sufficiently high value $va$.

Assumption (\ref{DM}) implies that $D^i({\cal M})\subseteq {\cal M}$ for
each $i\in\N$. We leave it to the reader to use this fact together with
(\ref{diffv}) to prove the following easy lemma by induction on $i$:
\begin{lemma}                               \label{lDk}
If $(K,D,v)$ admits asymptotic integration,
then for each $i\in\N$, the map
\begin{equation}                            \label{Dk}
D^i: {\cal M}\>\longrightarrow\>{\cal M}_{D^i}\>:=\>
\bigcup_{e\in {\cal M}} (D^i e){\cal O}\>\subseteq\>{\cal M}
\end{equation}
is an immediate embedding of ultrametric spaces with value map
$va\mapsto vD^i a$.
\end{lemma}
Hence by Theorem~\ref{MT}, we have:
\begin{lemma}
If $(K,D,v)$ is spherically complete and  admits asymptotic integration,
then the map (\ref{Dk}) is an isomorphism of ultrametric spaces.
\end{lemma}

When we try to prove a differential Hensel's Lemma for Rosenlicht's
differential valuations, we have to deal with the problem that the
connection between $vD^ia$ and $vD^ja$ for $i\ne j$ is not as nice as in
the case of $D$-fields. The natural hypothesis on the partial
derivatives as used in Theorem~\ref{DHL} may not suffice. We need to set
up a relation between the values $vy,vDy,\ldots,vD^ny$. The key is
definition (\ref{diffvdef}) of a differential valuation. By induction,
it implies that for arbitrary $e\in {\cal M}$,
\begin{equation}                            \label{vDky}
vD^i y+(n-i)vDe\>>\>vD^ny\qquad \mbox{for } 0\leq i<n\;.
\end{equation}
Because of this relation, we will have to assume that the partial
derivative of least value appears at the variable $X_n$ which is
associated with the highest power $D^n$ of~$D$. The following is a
special case of Theorem~\ref{genHLdom} in Section~\ref{sectdom}:

\begin{theorem}                             \label{DVHL}
Let $(K,D)$ be a differential field, endowed with a spherically complete
differential valuation $v$. Assume that $(K,D,v)$ admits asymptotic
integration. Take a polynomial $g\in {\cal O}[X_0,X_1,\ldots,X_n]$ and
assume that there are $b\in {\cal O}$ and $e\in {\cal M}$ such that,
with $d:=De$,
\[g(d^{-n}X_0,d^{1-n}X_1,\ldots,d^{-1}X_{n-1},X_n)
\in {\cal O}[X_0,X_1,\ldots,X_n]\]
and
\begin{equation}                            \label{DVHLc1}
v\frac{\partial g}{\partial X_n}(b,Db,\ldots, D^n b)
\>=\> \min_{0\leq i\leq n}vd^{i-n}\frac{\partial g}{\partial X_i}
(b,Db,\ldots,D^n b)\>=\>0
\end{equation}
and
\begin{equation}                            \label{DVHLc2}
vg^D b\>\geq\> vD^n e\;.
\end{equation}
Then there is a unique element $a\in {\cal O}$ such that
$g^D a=0$. It satisfies $v(a-b)\geq ve$.
\end{theorem}
\begin{proof}
Set $f(X_0,\ldots,X_n)=g(d^{-n}X_0,d^{1-n}X_1,\ldots,d^{-1}X_{n-1},
X_n)\in {\cal O}[X_0,X_1,\ldots,X_n]$. With $d_i$ defined as in
(\ref{d_i}) of Proposition~\ref{psi}, it follows from (\ref{DVHLc1})
that $0=vd_n=\min_i vd_i$, which shows that (\ref{dnmindi}) of
Proposition~\ref{immopdom} is satisfied. Further, we set $\sigma_i:=
d^{n-i}D^i$, $B:={\cal M}$ and $B':= {\cal M}_{D^n} \subseteq {\cal M}=
d_n{\cal M}$. Then by (\ref{vDky}), $v\sigma_n a< v\sigma_i a$ for all
$i<n$ and $a\in {\cal M}$, showing that (\ref{domin}) holds. Since
$\sigma_n ({\cal M}) =D^n({\cal M})\subseteq B'\subseteq {\cal M}$, it
follows that also $\sigma_i({\cal M})\subseteq {\cal M}$ for $0\leq
i\leq n$. Condition (\ref{DVHLc2}) tells us that condition
(\ref{genHLdomc}) of Theorem~\ref{genHLdom} is satisfied. Finally,
Lemma~\ref{lDk} tells us that $D^n: {\cal M} \rightarrow B'$ is
immediate and injective. We have proved that all conditions of
Theorem~\ref{genHLdom} are satisfied. Hence, there is a unique
element $a\in b+{\cal M}$ such that $g^D a=f(\sigma_0 a,
\sigma_1 a,\ldots,\sigma_n a)=0$, and it satisfies $v\sigma_n(a-b)\geq
v\sigma_n e$. The latter means that $vD^n(a-b)\geq vD^ne$, which by
(\ref{diffv}) implies $v(a-b)\geq ve$ since $a-b\,,\,e\in {\cal M}$.
\end{proof}

This theorem can be improved if one also considers the values of the
higher derivatives of $f$. The formal higher derivatives
$f^{[\,\ul{i}\,]}$ have already been introduced and used in
Section~\ref{sectbr}. We will work with the Rosenlicht system
\[\sigma_i\>:=\>D^i\,,\;\;\; e_i\>:=\>(De)^{n-i}\]
for fixed $n\in\N$ and some $e\in {\cal M}$. Then condition (\ref{vei})
in Section~\ref{sectRso} is trivially satisfied, and condition
(\ref{Rsys}) is satisfied because of (\ref{vDky}). We will apply
Theorem~\ref{genHLRos} to prove:

\begin{proposition}
Take $f\in {\cal O}[X_0,X_1,\ldots,X_n]$ and $b\in {\cal O}$ such that
\[d_n\>=\>\frac{\partial f}{\partial X_n}(b,Db,\ldots,D^n b)\>\ne\>0\]
and for all $\ul{i}\in I=\{0,\ldots,\deg f\}^{n+1}\setminus
\{(0,\ldots,0)\}$,
\begin{equation}                            \label{vfiD}
vf^{[\,\ul{i}\,]} (b,Db,\ldots,D^n b)\>\geq\>vd_n+ve_k\;\;\;\mbox{ if \ }
k= \min\{j\mid i_j \ne 0\}\>.
\end{equation}
Suppose further that for some balls $B,B'\subseteq {\cal M}$ around $0$,
$D^n: B\rightarrow B'$ is immediate. Then
\begin{equation}                            \label{embDi}
b+B\ni x\mapsto f^D x\in f^D b +d_n B'
\end{equation}
is an immediate embedding of ultrametric spaces with value map
$va\mapsto vd_n+vD^n a$. If $(K,v)$ is spherically complete, it is an
isomorphism of ultrametric spaces.
\end{proposition}
\begin{proof}
All this follows from Proposition~\ref{immopRos} and Theorem~\ref{MT}.
It just remains to prove that (\ref{embDi}) is an embedding of
ultrametric spaces with value map $va\mapsto vd_n+vD^n a$. But this
follows from equation (\ref{vmsig}) of Proposition~\ref{immopRos} and
the fact that $va\mapsto vD^na$ for $a\in {\cal M}$ preserves ``$<$''.
\end{proof}

\begin{theorem}                             \label{DVHLR}
Let $(K,D)$ be a differential field, endowed with a spherically complete
differential valuation $v$. Assume that $(K,D,v)$ admits asymptotic
integration. Take a polynomial $f\in {\cal O}[X_0,X_1,\ldots,X_n]$ and
assume that there are $b\in {\cal O}$ and $e\in {\cal M}$ such
that
\begin{equation}                            \label{vderiv}
\forall \ul{i}:\>
vf^{[\,\ul{i}\,]} (b,Db,\ldots,D^n b)\>\geq\>v\frac{\partial f}
{\partial X_n}(b,Db,\ldots, D^n b)+(n-k)vDe\;\;\mbox{ if \ }
k= \min\{j\mid i_j \ne 0\}
\end{equation}
and
\[vf^D b\>\geq\> vD^n e\;.\]
Then there is a unique element $a\in {\cal M}$ such that
$f(a,Da,\ldots,D^na)=0$. It satisfies $v(a-b)\geq ve$.
\end{theorem}
\begin{proof}
As in the proof of Theorem~\ref{DVHL}, we set $B:={\cal M}$ and $B':=
{\cal M}_{D^n}$; then $D^n: B\rightarrow B'$ is immediate by
Lemma~\ref{lDk}. Now we apply Theorem~\ref{genHLRos} instead of
Theorem~\ref{genHLdom}.
\end{proof}

If $K$ is of characteristic $0$, then the usual higher derivative
\[f^{(\,\ul{i}\,)}(X)\>:=\>\frac{\partial^{i_0+\ldots+i_n} f}
{\partial^{i_0}X_0\cdots\partial^{i_n}X_n}(X)\]
can be substituted for $f^{[\,\ul{i}\,]} (X)$ in the above theorem.
Indeed,
\[f^{(\,\ul{i}\,)} (X)\>=\>i_0!\cdot\ldots\cdot i_n!\cdot
f^{[\,\ul{i}\,]} (X)\]
and therefore,
\[vf^{(\,\ul{i}\,)} (b,Db,\ldots,D^n b)\>=\>
vf^{[\,\ul{i}\,]} (b,Db,\ldots,D^n b)\;.\]

\parb
In $\R((t))^{LE}$, the element $x=t^{-1}$ satisfies $vx<0$ and $Dx=1$.
Suppose that $1<r\in\R$. Then $e=\frac{1}{1-r}x^{1-r}\in
\R((t))^{LE}$ satisfies $ve>0$ and $De=x^{-r}$. With $K_i$ as in the
discussion at the end of Section~\ref{sectrvdf}, take
${\cal M}_i$ to be the valuation ideal of $K_i\,$. Then $\frac{1}{x}
\notin (De){\cal M}_i$ and it can be shown that for every $a'\in (De)
{\cal M}_i$ there is some $a\in e{\cal M}_i$ such that $v(a'-Da)>va'$.
As for the proof of Lemma~\ref{lDk}, it can thus be deduced that for
every $k\geq 1$, $D^k: e{\cal M}_i\rightarrow (D^k e){\cal M}_i$ is an
immediate embedding of ultrametric spaces. Hence on every ball of the
form $e{\cal M}_i$ in $K_i\,$, differential equations of the above
form can be solved without any modification of our approach.

The union of an ascending chain of henselian fields is again henselian.
With the same idea of proof, working in $K_i$ for all $i$ large enough
to contain all coefficients of $h$ and then passing to the union of the
$K_i\,$, one obtains, applying Theorem~\ref{DVHLR} with $e$ as given
above to the polynomial $f(X_0,\ldots,X_n)=g(X_0,\ldots,X_n)+c-X_n$ and
$b=0\,$:
\begin{theorem}                             \label{chform}
Let ${\cal O}$ denote the valuation ring of $\R((t))^{LE}$. Suppose
that
\begin{equation}                            \label{hform}
g(X_0,\ldots,X_n)\>\in\> \sum_{i=0}^{n-1}x^{-(n-i)r}X_i\,
{\cal O}[X_i\,,\,\ldots\,,\,X_n]\>+\>X_n^2{\cal O}\>+\>X_n{\cal M}
\end{equation}
and
\[c\in x^{-r-n+1}\,{\cal O}\;.\]
Then the differential equation
\begin{equation}
D^ny\>=\>g(y,Dy,\ldots,D^ny)\,+\,c
\end{equation}
has a unique infinitesimal solution in $\R((t))^{LE}$; this solution has
value $\geq vx^{1-r}$.
\end{theorem}

This theorem implies the following result, which was
proved by Lou van den Dries in~[D]:
\begin{corollary}                           \label{VDD}
Suppose that $p$ is a polynomial in one variable with coefficients
in $\R((t))^{LE}$, all of value $\geq vt^r$ for some $r\in\R$, $r>1$.
Then the differential equation
\[Dy\>=\>p(y)\]
has a unique infinitesimal solution in $\R((t))^{LE}$.
\end{corollary}

%
%
\section{Sums of spherically complete valued abelian groups}
                                            \label{sectgp}
Let $({\cal A},v)$ be a valued abelian group and $A_1,\ldots,A_n$ be
subgroups of ${\cal A}$. The restrictions of $v$ to every $A_i$ will
again be denoted by $v$. We call the sum $A_1+\ldots+A_n\subseteq
{\cal A}\,$ \bfind{pseudo-direct} if for every $a'\in A_1+\ldots+A_n\,$,
$a'\ne 0$, there are $a_i\in A_i$ such that
\begin{equation}                            \label{psdir}
v\sum_{i=1}^{n}a_i\> =\>\min_{1\leq i\leq n} va_i\;
\mbox{\ \ and\ \ }\; v\left(a'-\sum_{i=1}^{n}a_i\right)\> >\>va'\;.
\end{equation}

\begin{proposition}                         \label{pdstc}
The sum $A_1+\ldots+A_n\subseteq {\cal A}$ is pseudo-direct if and only
if the group homomorphism $f:\>A_1\times\ldots\times A_n\rightarrow
A_1+\ldots+A_n$ defined by $f(a_1,\ldots, a_n):= a_1+\ldots+a_n$ is
immediate.
\end{proposition}
\begin{proof}
$\Rightarrow:$ \ Assume that the sum $A_1+\ldots+A_n$
is pseudo-direct. Take any $a'\in \sum_{i} A_i$
and choose $a_i\in A_i$ such that (\ref{psdir}) holds. Then $a:=
(a_1,\ldots,a_n)\in A_1\times\ldots\times A_n$ satisfies (IH1).
If $b=(b_1,\ldots,b_n)\in A_1\times\ldots\times A_n$ such that
$vb\geq va$, then
\[vfb\>=\>v\sum_{i} b_i\>\geq\>\min_i vb_i\>=\>vb\>\geq\>va\>=\>
\min_i va_i \>=\>v\sum_{i} a_i\>=\>vfa\;.\]
This shows that $a$ also satisfies (IH2).

\mn
$\Leftarrow:$ \ Assume that $f$ is immediate. Take any $a'\in \sum_{i}
A_i\,$, $a'\ne 0$. Choose $a:= (a_1,\ldots,a_n)\in A_1\times\ldots\times
A_n$ such that (IH1) and (IH2) hold. Then $v\left(a'-\sum_i
a_i\right)=v(a'-fa)>va'$. Now choose some $j$ such that $va_j=\min_i
va_i\,$. Then set $b_j= a_j\in A_j$ and $b_i=0\in A_i$ for $i\ne j$. For
$b=(b_1,\ldots,b_n)$, we thus have that $va=\min_i va_i=va_j=vb_j=\min_i
vb_i=vb$. Hence by (IH2), $v\sum_i a_i= vfa\leq vfb=vb_j=\min_i
va_i\,$. We have proved that the elements $a_i$ satisfy (\ref{psdir}).
\end{proof}

If the groups $(A_i,v)$ are spherically complete, then by
Proposition~\ref{prodsphc}, the same is true for their direct product
$A:=A_1\times\ldots\times A_n\,$, endowed with the minimum
valuation as defined in (\ref{minval}). Hence, the foregoing
proposition, Theorem~\ref{MT} and Corollary~\ref{scoap} show:

\begin{theorem}                             \label{addgr}
Assume that the subgroups $(A_i,v)$ of $({\cal A},v)$, $1\leq i\leq
n$, are spherically complete. If the sum $A_1+\ldots+A_n$
is pseudo-direct, then it is also spherically complete and has the
optimal approximation property.
\end{theorem}

\bn\bn\bn
{\bf References}
\newenvironment{reference}%
{\begin{list}{}{\setlength{\labelwidth}{4em}\setlength{\labelsep}{0em}%
\setlength{\leftmargin}{3em}\setlength{\itemsep}{-1pt}%
\setlength{\baselineskip}{3pt}}}%
{\end{list}}
\newcommand{\lit}[1]{\item[{#1}\hfill]}
\begin{reference}
\lit{[D]} {van den Dries, L.$\,$: {\it Solutions of ODE's in
$\R((t))^{LE}$ as fixed points}, manuscript (November 1997)}
\lit{[DMM1]} {van den Dries, L.\ -- Macintyre, A.\ --
Marker, D.$\,$: {\it The elementary theory of restricted analytic
functions with exponentiation}, Annals Math.\ {\bf 140} (1994),
183--205}
\lit{[DMM2]} {van den Dries, L.\ -- Macintyre, A.\ --
Marker, D.$\,$: {\it Logarithmic-Exponen\-tial Power Series},
J.\ London Math.\ Soc.\ {\bf 56} (1997), 417--434}
\lit{[DMM3]} {van den Dries, L.\ -- Macintyre, A.\ -- Marker,
D.$\,$: {\it Logarithmic-Exponential Series}, Ann.\ Pure
Appl.\ Logic {\bf 111} (2001), 61--113}
\lit{[DS]} {van den Dries, L.\ -- Speissegger, P.$\,$: {\it The real
field with convergent generalized power series is model complete and
o-minimal}, Trans.\ Amer.\ Math.\ Soc.\ {\bf 350} (1998), 4377--4421}
\lit{[G]} {Gravett, K.~A.~H.$\,$: {\it Note on a result of Krull},
Cambridge Philos.\ Soc.\ Proc.\ {\bf 52} (1956), 379}
\lit{[KA]} {Kaplansky, I.$\,$: {\it Maximal fields with valuations I},
Duke Math.\ J.\ {\bf 9} (1942), 303--321}
\lit{[KR]} {Krull, W.$\,$: {\it Allgemeine Bewertungstheorie},
J.\ reine angew.\ Math.\ {\bf 167} (1931), 160--196}
\lit{[KU1]} {Kuhlmann, F.-V.: {\it A theorem on spherically complete
valued abelian groups}, The Fields Institute Preprint Series, Toronto
(June 1997)}
\lit{[KU2]} {Kuhlmann, F.-V.: {\it Elementary properties of power series
fields over finite fields}, J.\ Symb.\ Logic {\bf 66} (2001), 771-791}
\lit{[KU3]} {Kuhlmann, F.-V.$\,$: {\it Additive Polynomials and Their
Role in the Model Theory of Valued Fields}, in: Proceedings of
the Workshop and Conference on Logic, Algebra, and Arithmetic, Tehran,
Iran 2003, Lecture Notes in Logic, Vol.\ {\bf 13}, Association for
Symbolic Logic, AK Peters}
\lit{[KU4]} {Kuhlmann, F.-V.$\,$: Book in preparation. Preliminary
versions of several chapters available at:\n
{\tt http://math.usask.ca/$\,\tilde{ }\,$fvk/Fvkbook.htm}}
\lit{[L]} {Lang, S.$\,$: {\it Algebra}, Addison-Wesley, New York (1965)}
\lit{[P1]} {Prie{\ss}-Crampe, S.$\,$: {\it Angeordnete Strukturen.
Gruppen, K\"orper, projektive Ebenen}, Ergebnisse der Mathematik und
ihrer Grenzgebiete {\bf 98}, Springer (1983)}
\lit{[P2]} {Prie{\ss}-Crampe, S.$\,$: {\it Der Banachsche Fixpunktsatz
f\"ur ultrametrische R\"aume}, Results in Mathematics {\bf 18} (1990),
178--186}
\lit{[PR1]} {Prie{\ss}-Crampe, S.\ -- Ribenboim, P.$\,$: {\it Fixed
Points, Combs and Generalized Power Series},
Abh.\ Math.\ Sem.\ Hamburg {\bf 63} (1993), 227-244}
\lit{[PR2]} {Prie{\ss}-Crampe, S.\ -- Ribenboim, P.$\,$: {\it Fixed
Point and Attractor Theorems for Ultrametric Spaces}, Formum Math.\ {\bf
12} (2000), 53--64}
\lit{[PR3]} {Prie{\ss}-Crampe, S.\ -- Ribenboim, P.$\,$: {\it
Differential equations over valued fields (and more)},  J.\ Reine
Angew.\ Math.\ {\bf 576}  (2004), 123--147}
\lit{[RB]} {Robba, P.$\,$: {\it Lemmes de Hensel pour les op\'erateur
diff\'erentiels. Application a la r\'eduction formelle des \'equations
diff\'erentielles.}, Enseign.\ Math.\ (2) {\bf 26} (1980), 279--311}
\lit{[R1]} {Rosenlicht, M.$\,$: {\it Differential valuations},
Pacific J.\ Math.\ {\bf 86} (1980), 301--319}
\lit{[R2]} {Rosenlicht, M.$\,$: {\it On the value group of a
differential valuation}, Amer.\ J.\ Math.\ {\bf 191} (1979), 258--266}
\lit{[S1]} {Scanlon, T.$\,$: {\it Model theory of valued $D$-fields},
Ph.D.\ thesis, Harvard University (1997)}
\lit{[S2]} {Scanlon, T.$\,$: {\it A model complete theory of valued
D-fields}, J.\ Symb.\ Logic {\bf 65} (2000), 1758--1784}
\lit{[SCH]} {Sch\"orner, E.$\,$: {\it Ultrametric Fixed Point Theorems
and Applications}, in: Valuation Theory and its Applications, Volume II,
Proceedings of the Valuation Theory Conference Saskatoon 1999, eds.\
F.-V.\ Kuhlmann, S.\ Kuhlmann and M.\ Marshall, The Fields Institute
Communications Series {\bf 33}, Publications of the Amer.\ Math.\ Soc.\
(2003), 353--359}
\lit{[T]} {Teissier, B.$\,$: {\it Valuations, Deformations, and Toric
Geometry}, in: Valuation Theory and its Applications, Volume II,
Proceedings of the Valuation Theory Conference Saskatoon 1999, eds.\
F.-V.~Kuhlmann, S.~Kuhlmann and M.~Marshall, The Fields Institute
Communications Series {\bf 33}, Publications of the Amer.\ Math.\ Soc.\
(2003), 361--459}
\lit{[VC]} {Veluscek, D.\ -- Cimpric, J.$\,$: {\it Higher product
Pythagoras numbers of skew fields}, preprint. Available at:\n
{\tt http://www.fmf.uni-lj.si/srag/preprints/pitag$\underline{\ }$final1$\underline{\ }$popravki.ps}}
\end{reference}
\toradresse

\end{document}